\documentclass[11pt]{amsart}
\makeatletter
\@namedef{subjclassname@2020}{%
  \textup{2020} Mathematics Subject Classification}
\makeatother
\usepackage[utf8]{inputenc}
\usepackage[T1]{fontenc}
\usepackage{latexsym, amsmath, amssymb}
\usepackage{amsmath}
\usepackage{amsfonts}
\usepackage{amssymb}
\usepackage{amscd}
\usepackage{amsthm}

\usepackage{verbatim}
\usepackage[dvipsnames]{xcolor}
\usepackage{xypic}
\usepackage{tikz-cd}
\usepackage{latexsym}
\usepackage{todonotes}
\usepackage{mathrsfs}
\theoremstyle{plain}
\newtheorem{theoremint}{Theorem}
\newtheorem{theorem}{Theorem}[section]

\newtheorem{proposition}[theorem]{Proposition}
\newtheorem{lemma}[theorem]{Lemma}
\newtheorem{corollary}[theorem]{Corollary}
\theoremstyle{definition}
\newtheorem{definition}[theorem]{Definition}

\newtheorem{remark}[theorem]{Remark}

\numberwithin{equation}{section}

\DeclareMathOperator{\Ext}{Ext}

\DeclareMathOperator{\Hom}{Hom}

\newcommand{\arr}{\longrightarrow}
\newcommand{\rk}{\mathrm{rank}}

\newcommand{\vin}{\color{Bittersweet}}

\def\Ext{{\rm{Ext\,}}}
\def\Hom{{\rm{Hom\,}}}

\newcommand{\cc}{{\mathbb C}}
\newcommand{\pp}{{\mathbb P}}

\newcommand{\Pic}{\mathrm{Pic}}

\newcommand\sC{{\mathcal C}}
\newcommand\sD{{\mathcal D}}
\newcommand\sE{{\mathcal E}}
\newcommand\sF{{\mathcal F}}
\newcommand\sG{{\mathcal G}}
\newcommand\sH{{\mathcal H}}
\newcommand\sI{{\mathcal I}}

\newcommand\sK{{\mathcal K}}
\newcommand\sL{{\mathcal L}}

\newcommand\sN{{\mathcal N}}
\newcommand\sO{{\mathcal O}}

\newcommand\sQ{{\mathcal Q}}

\newcommand\sT{{\mathcal T}}

\newcommand\Sd{{S_{(1,d)}}}

\newcommand{\rvline}{\hspace*{-\arraycolsep}\vline\hspace*{-\arraycolsep}}

\def\pee#1{\hbox{$ {\mathbb P}^{#1}$}}

  \def \tab#1{\kern #1 truein}

 \newenvironment{sistema}%
  {\left\lbrace\begin{array}{@{}l@{}}}%
  {\end{array}\right.}

\title['t Hooft bundles on the complete flag threefold]{'t Hooft bundles on the complete flag threefold and moduli spaces of instantons}
\author{V. Antonelli, F. Malaspina, S. Marchesi and J. Pons-Llopis}

\address{Politecnico di Torino, Corso Duca degli Abruzzi 24, 10129 Torino, Italy}
\email{vincenzo.antonelli@polito.it}

\address{Politecnico di Torino, Corso Duca degli Abruzzi 24, 10129 Torino, Italy}
\email{francesco.malaspina@polito.it}

\address{Universitat de Barcelona, Gran Via de les Corts Catalanes, 585, 08007 Barcelona, Spain \\ 
Centre de Recerca Matem\`atica Edifici C, Campus Bellaterra, 08193 Bellaterra, Spain}
\email{marchesi@ub.edu}

\address{Politecnico di Torino, Corso Duca degli Abruzzi 24, 10129 Torino, Italy}
\email{juan.ponsllopis@polito.it}

\thanks{VA, FM and JPL are members of the GNSAGA group of INdAM. SM is partially supported by PID2020-113674GB-I00 and by the Spanish State Research Agency, through the Severo Ochoa and Mar\'ia de Maeztu Program for Centers and Units of Excellence in R\&D (CEX2020-001084-M)}

 
\begin{document}

\begin{abstract}
In this work we study the moduli spaces of instanton bundles on the flag twistor space $F:=F(0,1,2)$. We stratify them in terms of the minimal twist supporting global sections and we introduce the notion of (special) 't Hooft bundle on  $F$. In particular we prove that there exist $\mu$-stable 't Hooft bundles for each admissible charge $k$. We  completely describe the geometric structure of the moduli space of (special) 't Hooft bundles for arbitrary charge $k$. Along the way to reach these goals, we describe the possible structures of multiple curves supported on some rational curves in $F$ as well as the family of del Pezzo surfaces realized as hyperplane sections of $F$. Finally we investigate the splitting behaviour of 't Hooft bundles when restricted to conics.
\end{abstract}

\maketitle

\section{Introduction}\label{sIntro}

Arguably, one of the most important incentives that spurred research in the field of algebraic vector bundles came from Yang-Mills theory, an a priori unrelated area. Arising from gauge theory for non-abelian groups whose aim was to provide an explanation of weak and strong interactions, the original Yang-Mills theory coined the term instanton to 
denote the minimum action solutions of the Yang-Mills equations on the $4$-sphere. In terms of differential geometry instantons are connections with self-dual curvature on a smooth $SU(2)$-bundle  $\mathcal{E}$ over $S^4$.

Identifying $S^4$ with the quaternionic projective line $\mathbb P^1(\mathbb{H})$, twistor theory, as it was developed by R. Penrose, permitted to encode the  differential geometry properties of $S^4$ in terms of holomorphic data of its associated twistor space $\pi:\mathbb P^3(\mathbb{C})\rightarrow \mathbb P^1(\mathbb{H})\cong S^4$. In particular, pulling back a self-dual curvature on $\mathbb{P}^3$ by means of the twistor projection $\pi$ defines a holomorphic structure on the bundle $\pi^*\sE$. Atiyah and Ward  (cf. \cite{AW}) realized that indeed it is possible to recover the original instanton connection from the holomorphic structure on $\pi^\ast\sE$. Motivated by this correspondence, a (mathematical) instanton bundle of charge $k>0$ on $\mathbb P^3$ was defined as a stable rank two algebraic vector bundle $\sF$ with $c_1(\sF)=0$, $c_2(\sF)=k$ and $H^1(\sF(-2))=0$. Atiyah-Ward correspondence states therefore that there exists a bijection between the instanton connections on $S^4$ and mathematical instanton bundles on $\mathbb P^3$ with some extra real conditions. Once it was realized by Barth that the $k$-instanton bundles are exactly the cohomology sheaves of monads of the form
$$
0\rightarrow\sO_{\mathbb P^3}(-1)^{\oplus k}\rightarrow\sO_{\mathbb P^3}^{\oplus 2k+2}\rightarrow\sO_{\mathbb P^3}(1)^{\oplus k}\rightarrow 0
$$
the problem of classifying instantons became mostly a problem of linear algebra and from this point of view it was finally settled in the seminal paper \cite{ADHM}.

Prompted by this exciting set of results, algebraic geometers embarked on the study of the geometric properties of the moduli space $MI_{\mathbb P^3}(k)$ of $k$-instanton bundles, seen as a subspace of the moduli space $M^s_{\mathbb P^3}(2;0,k)$ of the Maruyama moduli space of rank two stable bundles with Chern classes $c_1=0$ and $c_2=k$. This contributed to the development of many techniques in the theory of vector bundles (e.g. monads, jumping rational curves, Serre correspondence) that by now have asserted themselves as crucial tools in the area.

However, despite the progress and use of such a strong machinery, the full understanding of $MI_{\mathbb P^3}(k)$ had revealed itself an extremely difficult issue. Indeed, it was after a tour de force sprawled along four decades that the main geometric properties of $MI_{\mathbb P^3}(k)$ were determined: it is an irreducible (cf. \cite{T1} and \cite{T2})) smooth (cf. \cite{JV}), affine (cf. \cite{CO}) variety of dimension $8k-3$.

Alongside this exciting line of research, the definition of a (mathematical) instanton bundle has been largely generalized, either by considering other projective varieties as the supporting space of the vector bundle or by relaxing the conditions on the vector bundle itself. In the former situation, instanton bundles have been defined and studied  recently for an arbitrary Fano threefold with Picard number one (cf. \cite{Fa2} and \cite{Kuz}) and later on for arbitrary projective varieties (cf.  \cite{AnCa} and \cite{AnMa}).  In the latter situation,  perverse instanton sheaves have been singled out and studied in the setting of derived categories (cf. \cite{CJMM}).

However, there has been a thread of research in this field, closely related to the original motivation, that, in our opinion, did not receive the attention it deserved. Indeed, by means of a theorem by Hitchin (cf. \cite{Hit}), there  only exist two projective varieties on which the link between instantons from the point of view of differential geometry and those from algebraic geometry can be done meaningfully: the projective space $\mathbb P^3$ and the threefold $F$ of point-line flags in $\mathbb P^2$. More precisely, Hitchin showed that the only twistor spaces of four dimensional (real) differential varieties which are K\"ahler (and a fortiori, projective) are $\mathbb P^3$ and the flag variety $F$, which is the twistor space of $\mathbb P^2$. Whereas the case of $\mathbb P^3$, as we have pointed out, has been thoroughly studied and led to many breakthroughs in Algebraic Geometry, much less work has been devoted to the flag threefold $F$ (cf. \cite{Bu}). For this particular projective variety $F$, the relevant definition of (mathematical) instanton bundle $\sE$  of charge $k$ on $F$ is the following: $\sE$ is a rank two vector bundle such that $c_1(\sE)=0$, $c_2(\sE)=kh_1h_2$, $h^1(\sE(-h_1-h_2))=0$, $h^0(\sE)=0$ and $\sE$ is $\mu$-semistable, where $h_i$ are the pullbacks of the class of a line in $\mathbb{P}^2$ under the two natural projections. Therefore, we intend with this paper to contribute, following a former work of some of the authors (cf. \cite{MMP}), to a full understanding of the moduli space of instanton bundles on the flag variety $F$.

A natural way to study this object is by stratifying it according to the first twist under which a given instanton bundle $\sE$ has global sections. However, since $\Pic(F)\cong \mathbb{Z}^{\oplus 2}\cong \langle h_1,h_2\rangle$, this stratification turned out to be more involved than in the well-studied case of $\mathbb P^3$. Therefore, trying to keep the enlightening analogy with the well-known case of instanton bundles on $\mathbb P^3$ (cf. \cite{Ha3}, \cite{HN} and \cite{BeFr}), we define \textit{$D$-'t Hooft instantons} as the instantons $\sE$ such that $h^0(\sE(D))\neq 0$ for an effective divisor $D$. We are particularly interested in instantons acquiring global sections after the lowest possible twists, namely $D=h_i$ or $D=h_1+h_2$. Among them, 't Hooft instantons for which the zero locus of a minimal global section lies on a hyperplane surface section of $F$ will be called \textit{special 't Hooft instantons}. Hence, in order to perform our task, we need to develop a careful study of different families of surfaces and curves relevant to our goals and that, in any case, we believe it is of interest on its own. In the particular case of surfaces, we give a detailed account of the structure of the Hilbert scheme of hyperplane sections of $F$, relevant for the understanding of special 't Hooft bundles. They turn out to stand for a large and interesting family of del Pezzo surfaces of degree $6$.

On the other hand, in the case of curves inside $F$, we were leaded to a careful study of some subtleties about the possible multiple structures on particular families of rational curves living in $F$ that are in correspondence with the zero loci of global sections of twists of 't Hooft bundles. 
As a first step towards our program, we prove several existence results (see Proposition \ref{serreexistenceflag}, Corollary \ref{htHooft} and Corollary \ref{cExistenceSpecial}), which we collect in the following theorem.

\begin{theoremint}
    For each $k\geq 1$ there exist $\mu$-stable $h_i$-'t Hooft bundles and special instanton bundles of charge $k$. Moreover, for each $k\geq 2$ there exist $\mu$-stable, proper $(h_1+h_2)$-'t Hooft bundles.
\end{theoremint}

Once the existence of such $\mu$-stable locally free sheaves is settled, we investigate their parameter spaces by means of their associated curves. Let us denote by $MI(k)$ the moduli spaces of instanton bundles of charge $k$, i.e. the open subset of $\mu$-stable instanton bundles inside the Maruyama moduli space $M_F(2;0,kh_1h_2)$ of rank two $\mu$-stable bundles with $c_1=0$ and $c_2=kh_1h_2$. Furthermore we denote by $MI_s(k)\subset MI(k)$ (resp. ${MI^i}(k)\subset MI(k)$) the closed subset of special (resp. $h_i$-'t Hooft) instanton bundles.
A study of associated curves allowed us to determine the geometric properties of the moduli spaces $MI_s(k)$ and $MI^i(k)$ (see Theorem \ref{tModuliSpecial} and Theorem \ref{tModuliHooft}):

\begin{theoremint}
     For any $k\ge 2$, the moduli space ${MI_s}(k)$ consists of two irreducible, smooth components ${MI_s'}(k)$ and ${MI_s''}(k)$ of dimension $7+2k$ and $4k+4$. The moduli space ${MI^i}(k)$ is a smooth variety consisting of at least $k$ irreducible components of dimension $5k+2$. Moreover, ${MI_H}(k):= MI^1(k)\cup MI^2(k)$ is singular precisely along the intersection $MI_s(k)$. 
\end{theoremint}

Let us outline now the contents of this paper. In Section \ref{sGeometryFlag} we introduce the main properties of the geometry of the flag variety $F$ that will be useful for our research. In Section \ref{sDescriptionDelPezzo} we focus our attention on the hyperplane sections of $F$, since they will be crucial to understand the moduli space of special 't Hooft instanton bundles. They are del Pezzo surfaces of degree $6$ but they can be smooth, singular or even reducible. So an accurate study of their properties and their Hilbert scheme will be carried out. In Section \ref{sMultipleCurves} we pay our attention to the kind of curves that will correspond exactly to the zero loci of sections of 't Hooft bundles. For doing this, it will be necessary to develop a careful study of multiple structures on some particular rational curves. In Section \ref{sHitHooft} we introduce the main characters of this paper, namely instanton bundles on the flag threefold. We also stratify them in terms of the first twist for which they have global sections and, in particular, we define and study 't Hooft instantons. In Section \ref{sHtHooft} we carry out a detailed study of instanton bundles $\sE$ such that $h^0(\sE(h_1+h_2))\neq 0$ by means of the elliptic curves associated to these sections. In Section \ref{sSpecialtHooft}, we define, trying to follow the analogy with the case of instantons on the projective space $\mathbb{P}^3$, special 't Hooft instantons. In our situation, they have the particular property of being 't Hooft instantons whose associated zero loci of the global sections under the minimal twist are contained in del Pezzo hyperplane sections. The previous work leads in Section \ref{sModulitHooft} to the main result of this paper, namely the description of the main geometric properties of the moduli space of special 't Hooft instanton bundles of any charge. Finally, we conclude this paper in Section \ref{sJumpingConics} with a description of the behaviour of 't Hooft instanton bundles restricted to smooth conics.

\noindent\textbf{Acknowledgments:} The authors  want to  deeply thank the referee for his/her careful reading, which led to an improvement of the paper.

\section{The geometry of the flag variety}\label{sGeometryFlag}

In this section, we will recall the relevant definitions and results on the flag variety $F$, defined by the point--line incidence 
$$
F:=\{(p,L)\in \mathbb P^2\times\mathbb P^{2\vee} \mid p\in L\}\subset \mathbb P^2\times\mathbb P^{2\vee}.
$$
For more details, we advise the reader to consult \cite[Section 2]{MMP}.
The peculiarity of $F$ relies, among other things, on the fact that it admits several different geometric descriptions.

For example, it is possible to construct $F$ as the general hyperplane section of $\mathbb P^2\times\mathbb P^{2\vee}$. We may suppose that $F$ is realized as the zero locus of the bihomogeneous equation
\[
x_0y_0+x_1y_1+x_2y_2=0
\]
in the coordinates of $\mathbb P^2\times\mathbb P^{2\vee}$. This point of view allows us to describe the flag variety as the projectivization of (a twist) of the cotangent bundle of $\pp^2$ and denoting by $\pi_i\colon F\to\mathbb P^2$ the restrictions of the natural projections, we see that they coincide with the canonical maps (one for each projective plane in the product) $\mathbb P(\Omega_{\mathbb P^2}^1(2))\to\mathbb P^2$.\\
Let $A(F)$ be the Chow ring of $F$ and $h_i= \pi_i^*\sO_{\mathbb P^2}(1)$, for $i=1,2$, the classes in $A^1(F)$ given by the two hyperplane divisors. From now on, if not explicitly specified, $h_i$ will denote either $h_1$ or $h_2$ and if both $h_i$ and $h_j$ appear, we will assume $i\neq j$. Then, it is possible to describe the Chow ring as 
$$
A(F)\cong A(\mathbb P^2)[h_1]/(h_1^2-h_1h_2+h_2^2)\cong \mathbb Z[h_1,h_2]/(h_1^2-h_1h_2+h_2^2, h_1^3,h_2^3).
$$
In particular, $\Pic(F)\cong\mathbb Z^{\oplus2}$, with generators $h_1$ and $h_2$, and $h=h_1 + h_2$ will represent the class of the hyperplane section of $F$.

We will now recall how to compute the cohomology of the line bundles on $F$:

\begin{lemma}\cite[Proposition 2.4]{MMP}
\label{pLineBundle}
For each $\alpha_1,\alpha_2\in\mathbb Z$ with $\alpha_1\le \alpha_2$, we have 
$$
h^i\big(\sO_F(\alpha_1h_1+\alpha_2h_2)\big)\ne0
$$
if and only if one of the following occurs:
\begin{itemize}
\item $i=0$ and $\alpha_1\ge0$;
\item $i=1$ and $\alpha_1\le -2$, $\alpha_1+\alpha_2+1\ge0$;
\item $i=2$ and $\alpha_2\ge0$, $\alpha_1+\alpha_2+3\le0$;
\item $i=3$ and $\alpha_2\le -2$.
\end{itemize}
In all these cases
$$
h^i\big(\sO_F(\alpha_1h_1+\alpha_2h_2)\big)=(-1)^i\frac{(\alpha_1+1)(\alpha_2+1)(\alpha_1+\alpha_2+2)}{2}.
$$
\end{lemma}

Let us make explicit, for the reader's convenience, the Riemann--Roch formula on the flag variety $F$.
\begin{equation}\label{eulerchar}
\chi(\sE) = r + \frac{3}{2} c_1 h_1 h_2 + \frac{1}{2}(c_1^2 - 2 c_2)(h_1+h_2) + \frac{1}{6}(c_1^3 - 3 c_1c_2 + 3c_3),
\end{equation}
where $c_i=c_i(\sE)$. In particular, for a rank $2$ vector bundle $\sE$ with Chern classes $c_1(\sE)=0$ and $c_2(\sE)=kh_1h_2$, using that
$$
c_1(\sE(ah_1+bh_2))=c_1(\sE)+2ah_1+2bh_2 \quad \text{and} \quad \
c_2(\sE(ah_1+bh_2))=c_2(\sE)+ (ah_1+bh_2)^2,
$$
\noindent after some easy but tedious computations, we obtain:
\begin{equation} \label{euler-general-twist}
\chi(\sE(ah_1+bh_2))=a^2b+ab^2+a^2+b^2+4ab+3a+3b+ 2-k(2+a+b).
\end{equation}

We will now focus on particular curves and surfaces that will appear in the following sections. 

Let us start by recalling that the flag variety $F$ contains two families of lines $\Lambda_1$,$\Lambda_2$, each isomorphic to $\pp^2$. Their representatives in the Chow ring $A(F)$ are $h_1^2$, $h_2^2$. Notice that if we look at $F$ as the projective bundle $\mathbb P(\Omega_{\mathbb P^2}^1(2))\to\mathbb P^2$, these families correspond to the fibers over points of $\pp^2$. We have a geometrical description: given $p\in\pp^2$, $\lambda_p:=\{L\in\pp^{2\vee}\mid p\in L\}\in\Lambda_1$. Analogously, given a line $L\subset\pp^2$, $\lambda_L:=\{p\in\pp^2\mid p\in L \}\in\Lambda_2$. Notice that $\lambda_x\cap\lambda_y=\emptyset$ if $x\neq y$ and $\lambda_x\cap\lambda_L=\emptyset$ (resp. $\lambda_x\cap\lambda_L=\{(x,L)\}$) if $x\notin L$ (resp. $x\in L$).
If $L_1$ (resp. $L_2$) is a line from the family $\Lambda_1$ (resp. $\Lambda_2$), it holds that
$$
\sO_F(\alpha h_1+\beta h_2)\otimes\sO_{L_1}\cong\sO_{\pp^1}(\beta) \quad\text{(resp.}\ \sO_F(\alpha h_1+\beta h_2)\otimes\sO_{L_2}\cong\sO_{\pp^1}(\alpha))
$$
\noindent since $h_1^2(\alpha h_1+\beta h_2)=\beta h_1^2h_2$ (resp. $h_2^2(\alpha h_1+\beta h_2)=\alpha h_1^2h_2$).

The $\sO_F$-resolution of a line $L_i$ is:
\begin{equation}\label{res-line}
  0 \longrightarrow \sO_F(-2h_i) \longrightarrow  \sO_F(-h_i)^{\oplus 2} \longrightarrow \sO_F \longrightarrow \sO_{L_i} \longrightarrow 0;
\end{equation}

The flag variety $F$ also contains a family of conics $C$ whose $O_F$-resolution is:

\begin{equation}\label{res-conic}
  \begin{array}{rcl}
& \sO_F(-h_1) \\
0 \longrightarrow \sO_F(-h) \longrightarrow  & \oplus & \longrightarrow \sO_F \longrightarrow \sO_C \longrightarrow 0.\\
& \sO_F(-h_2)
\end{array}
\end{equation}

It is possible to describe this family as in the following lemma.
\begin{lemma}\cite[Lemma 2.5]{MMP}
The Hilbert scheme of rational curves of degree two $\mathscr{C}:=\mathrm{Hilb}^{2t+1}(F)$ is isomorphic to $\pp^2\times\pp^{2\vee}$. The open set $\pp^2\times\pp^{2\vee}\backslash F$ corresponds to smooth conics. Moreover, the canonical map $p:\sC\rightarrow F$ from the universal conic $\sC$ to $F$ endows $\sC$ with the structure of a quadric bundle of relative dimension $2$ over $F$.
\end{lemma}

The relevance of conics is also motivated by the following result.
\begin{lemma}\label{l2points}\cite[Proposition 2.8]{MMP}
  Given two non-aligned points of $F$, there exists exactly one smooth conic passing through them.
\end{lemma}

In the remaining part of this section we will explicitly describe some noticeable surfaces contained in $F$. We use \cite{ABBS} as general reference. We are particularly interested in surfaces in the linear systems $|h_1+dh_2|$ with $d \ge 0$ (the case $dh_1+h_2$ being completely symmetric). Let $S_{(1,d)}$ be a surface with class $h_1+dh_2$ in the Chow ring. Via the second projection $\pi_2$, $S_{(1,d)}$ has a natural structure of blowup in $q=d^2+d+1$ points.

\begin{itemize}
    \item $S_{(1,0)} \in |h_1|$.\\
    In this case $S_{(1,0)}$ is isomorphic to a cubic scroll in $\pp^4$. It can be viewed as the blowup of $\mathbb{P}^{2\vee}$ at one point via the projection $\pi_2$. Since $S_{(1,0)}$ is the pullback of a line on the first $\pp^2$ factor, $S_{(1,0)}$ is also isomorphic to the Hirzebruch surface $\mathbb{F}_1=\pp(\sO_{\pp^1} \oplus \sO_{\pp^1}(-1))\to \pp^1$ embedded via the very ample line bundle $\sO_{\mathbb{F}_1}(C_0+2f)$ where $C_0$ is the negative self-intersection section and $f$ is a fiber.
    \smallskip
    \item $S_{(1,1)} \in |h_1+h_2|$.\\
    Using the adjunction formula, it is straightforward to see that $K_{S_{(1,1)}}\cong \sO_{S_{(1,1)}}(h)$ where $h$ is the restriction on $S_{(1,1)}$ of the hyperplane section of $F$. Thus $S_{(1,1)}$ is a degree six del Pezzo surface. If it is irreducible, then it is isomorphic to the blowup of $\pp^2$ in three (possibly infinitely near) points. Otherwise $S_{(1,1)}$ is the union of two Hirzebruch surfaces $S_{(1,0)}$ and $S_{(0,1)}$.
    \smallskip
    \item $S_{(1,d)} \in |h_1+dh_2|$ with $d\ge 2$.\\
    This surface can be seen as the blowup of $\mathbb{P}^{2\vee}$ in $q=d^2+d+1$ (possibly infinitely near) points via the second projection $\pi_2$. The hyperplane section is given by the very ample line bundle $\sO_{S_{(1,d)}}\bigl((d+2)l-\sum_{i=1}^{q}{e_i}\bigr)$, where $l$ represents the pullback of a line from $\mathbb{P}^{2\vee}$ and $e_i$ represent the exceptional divisors. 
\end{itemize}
In particular notice that $\Sd$ is a rational surface for each $d \ge 0$. In the following lemma we compute the restriction of the Picard generators $h_1$ and $h_2$ of $F$ to the surfaces $S_{(1,d)}$.

\begin{lemma}\label{lRestrictionPic}
Let $\Sd$ be a smooth surface in the linear system $|h_1+dh_2|$ and let $q=d^2+d+1$. Consider $\Sd\cong \mathrm{Bl}_{Z}(\pp^{2\vee})\xrightarrow{\pi_2}\pp^{2\vee}$ with $Z$ a $0$--dimensional subscheme of $q$ distinct points. Then the restriction map 
$$
\phi:\mathrm{Pic}(F)\cong \mathbb{Z}^{\oplus 2}\langle h_1,h_2\rangle\longrightarrow \Pic(\Sd)\cong \mathbb{Z}^{\oplus (q+1)}\langle l,e_1,\dots,e_q\rangle
$$
is completely determined by 
\[
\phi(h_1)=(d+1)l-\sum_{i=1}^{q}{e_i} \quad \text{and} \quad \phi(h_2)=l.
\]
Moreover 
\begin{itemize}
    \item if $d=0$, then $S_{(1,0)}=\mathbb{F}_1 \xrightarrow{\pi_1} \pp^1 \subset \pp^2$ and the restriction map 
$$\psi:\mathrm{Pic}(F)\cong \mathbb{Z}^{\oplus 2}\langle h_1,h_2\rangle\longrightarrow \Pic(S_{(1,0)})\cong \mathbb{Z}^{\oplus 2}\langle C_0,f\rangle$$
satisfies $\psi(h_1)=f$ and $\psi(h_2)=C_0+f$;
    \item if $d=1$, then $S_{(1,1)}$ it is also isomorphic to $\mathrm{Bl}_{Z}\xrightarrow{\pi_1}\pp^2$ via the first projection, where $Z$ is a $0$-dimensional subscheme of $3$ distinct points. The restriction map
    $$\psi:\mathrm{Pic}(F)\cong \mathbb{Z}^{\oplus 2}\langle h_1,h_2\rangle\longrightarrow \Pic(S_{(1,1)})\cong \mathbb{Z}^{\oplus 4}\langle l,e_1,e_2,e_3\rangle $$
    satisfies $\psi(h_1)=l$ and $\psi(h_2)=2l-e_1-e_2-e_3$.
\end{itemize}
\end{lemma}
\begin{proof}
In order to prove the statement, we first deal with the general case. Notice that $\sO_{\Sd}(h)\cong \sO_F(h_1+h_2) \otimes \sO_{\Sd}$ thus \begin{equation}\label{eHyperplaneBl}
\phi(h_1+h_2)=(d+2)l-\sum_{i=1}^{q}e_i.
\end{equation}
The restriction of $h_2$ to the surface $\Sd$ is a curve with class $h_1h_2+dh_2^2$, thus its image via the second projection is a line. Since any curve in the class $h_1h_2+dh_2^2$ has degree $d+2$, the only possibility is $\phi(h_2)=l$. Using linearity and \eqref{eHyperplaneBl} one obtains $\phi(h_1)=(d+1)l-\sum_{i=1}^{q}{e_i}$.

If $d=0$ and $S_{(1,0)}=\mathbb{F}_1 \xrightarrow{\pi_1} \pp^1 \subset \pp^2$, then $\phi(h_1)$ is a fiber of the $\pp^1$--bundle, and the statement follows by linearity from the equality $\phi(h_1+h_2)=C_0+2f$.

Finally the case $d=1$ can be obtained in a completely analogous way as in the general case by considering the first projection $\pi_1$. 
\end{proof}

Notice that the previous lemma describes the structure of the Picard group of smooth, irreducible, degree six del Pezzo surfaces $S_{(1,1)}$. However in Section \ref{sSpecialtHooft}, we will also deal with singular, irreducible ones.

\begin{lemma}\label{lSingDelPezzo}
Let $S:=S_{(1,1)}$ be a singular, irreducible del Pezzo surface. Then we have the following two possibilities:
\begin{itemize}
    \item $S\cong \mathrm{Bl}_{Z}(\pp^{2\vee})\xrightarrow{\pi_2}\pp^{2\vee}$ with $Z$ corresponding to the bubble configuration $p_2\succ p_1$, $p_3$. In particular $S$ has an $A_1$-type singularity and the restriction map $$\phi:\mathrm{Pic}(F)\cong \mathbb{Z}^{\oplus 2}\langle h_1,h_2\rangle\longrightarrow \mathrm{Cl}(S)\cong \mathbb{Z}^{\oplus 3}\langle l,f,g\rangle $$
    satisfies $\phi(h_1)=2l-2f-g$ and $\phi(h_2)=l$, where $f$ and $g$ represent the exceptional divisors.
    \medbreak
    \item $S\cong \mathrm{Bl}_{Z}(\pp^{2\vee})\xrightarrow{\pi_2}\pp^{2\vee}$ with $Z$ corresponding to the bubble configuration $p_3\succ p_2\succ p_1$. In particular $S$ has an $A_2$-type singularity and the restriction map $$\phi:\mathrm{Pic}(F)\cong \mathbb{Z}^{\oplus 2}\langle h_1,h_2\rangle\longrightarrow \mathrm{Cl}(S)\cong \mathbb{Z}^{\oplus 2}\langle l,g\rangle$$
    satisfies $\phi(h_1)=2l-3g$ and $\phi(h_2)=l$, where $g$ represents the exceptional divisor. 
\end{itemize} 
An identical statement holds considering the first projection $\pi_1$, the restriction map $\psi$ as in Lemma \ref{lRestrictionPic} and swapping the roles of $h_1$ and $h_2$. 
\end{lemma}
\begin{proof}
 According to the list of singular del Pezzo surfaces from \cite[Section 8.4.2]{Dol} and the description of these surfaces appearing as hyperplane section of $F$ (cf. \cite[Section 4]{ABBS}), the surfaces in the statement are the only irreducible, singular del Pezzo surfaces contained in $F$. Let us start with the first one, the case where $S$ has a unique singular point of type $A_1$. The minimal resolution $S'$ of singularities of $S$ is a weak del Pezzo surface corresponding to the bubble configuration $p_2\succ p_1, p_3$. In this case $\Pic(S')\cong \mathbb{Z}^{\oplus 4}\langle l, e, f, g\rangle$ with $l^2=1$, $e^2=-2$, $f^2=-1$, $g^2=-1$, $ef=1$, $eg=fg=le=lf=lg=0$. The desingularization map is the contraction of the $(-2)$--curve $e$, thus $\mathrm{Cl}(S)\cong \mathbb{Z}^{\oplus 3}\langle l, f, g\rangle$ and $f^2=-1/2$. Moreover we have $K_{S'}=-3l+e+2f+g$ and $-H_S=K_S=-3l+2f+g$. The part of the statement regarding the restriction map follows, as in Lemma \ref{lRestrictionPic}, from the fact that $\phi(h_1+h_2)=H_S$ and $(h_2)_{|S}$ projects to a line via the second projection $\pi_2$.

In the second case $S$ has a unique singular point of type $A_2$. The minimal resolution $S'$ of the singularity of $S$ is  a weak del Pezzo surface corresponding to the bubble configuration
 $x_3\succ x_2\succ x_1$. In this case $\Pic(S')\cong\mathbb{Z}^{\oplus 4} \langle l, e, f, g\rangle$ with $l^2=1$, $e^2=-2$, $f^2=-2$, $g^2=-1$, $ef=fg=1$, $eg=le=lf=lg=0$. The desingularization map is the contraction of the two $(-2)$--curves $e$ and $f$, thus $\mathrm{Cl}(S)\cong\mathbb{Z}^{\oplus 2} \langle l, g\rangle$ with the  relation $g^2=-1/3$. Moreover $K_{S'}=-3l+e+2f+3g$ and $-H_S=K_{S}=3l-3g$. The restriction map is obtained as in the proof of the previous point. 
\end{proof}

\section{Description of the Hilbert scheme of degree six del Pezzo surfaces in $F$}\label{sDescriptionDelPezzo}
In this section we describe the space of degree six del Pezzo surfaces contained in the flag variety $F$. We start with a local description, by which we mean that, for each del Pezzo surface considered, we explicitly write an associated matrix that allows the study of its local deformations. This will be used in the proof of Theorem \ref{tModuliSpecial}, main result of Section \ref{sModulitHooft}. Thereafter, we will focus on the global structure of the Hilbert scheme of this kind of surfaces,  obtaining a complete  description of the loci of smooth and singular surfaces.
\subsection{A local description}\label{sLocalDescription}
In the following part, we give an explicit presentation of all the degree six del Pezzo surfaces listed in Lemmas \ref{lRestrictionPic} and \ref{lSingDelPezzo}. This will be of extreme importance when we will deform (see Section \ref{sModulitHooft}) the configurations of curves associated to an instanton bundle.
 
Any $S_{(1,1)}$ can be defined, in the product $\mathbb{P}^2 \times \mathbb{P}^{2\vee}$, by the system of equations of the form
\begin{equation}\label{eq-defS6}
\left\{
\begin{array}{l}
x_0y_0 + x_1 y_1 + x_2y_2 = 0,\\
\\
\displaystyle \sum_{\substack{0\le i,j \le 2}} a_{i,j}x_iy_j=0.
\end{array}
\right.
\end{equation}
The first equation defines the flag variety in $\pp^2 \times \pp^{2\vee}$, while the second one determines its hyperplane section $S_{(1,1)}$. Let us fix a point $p$ in the first projective plane; substituting its coordinates in the system (\ref{eq-defS6}), we see that the two linear forms obtained in the $y_i$'s are linearly dependent if and only if $(\pi_1)^{-1}_{\mid S_{(1,1)}}(p) \simeq \pp^1$, i.e. if and only if $p$ is either one of the blown up points of the plane or a point of the line that gives us the fibration in the reducible case. In fact, a point $p =(x_0:x_1:x_2)$ gives linear dependent forms if and only if 
\begin{equation}\label{mat-det-points}
\rk
\left(
\begin{array}{ccc}
x_0 & x_1 & x_2\\
a_{0,0} x_0 + a_{1,0} x_1 + a_{2,0} x_2 & a_{0,1} x_0 + a_{1,1} x_1 + a_{2,1} x_2 & a_{0,2} x_0 + a_{1,2} x_1 + a_{2,2} x_2
\end{array}
\right) < 2.
\end{equation}
Observe that the previous matrix is constructed taking, for each column, the coefficients of the (1,1)-forms that appear in  (\ref{eq-defS6}) considered in the variables $y_0, y_1$ and $y_2$ respectively. It is know that, for the general choice of the coefficients $a_{i,j}$, the associated determinantal variety is exactly three non-aligned points of the projective plane.

We will now make explicit the coefficients $a_{i,j}$ for the remaining cases, i.e., the singular del Pezzo surfaces (either irreducible or reducible). Consider $S_{(1,1)}$ irreducible with a $A_1$-type singularity. Recall that such a surface can be constructed considering, as a first step, the blowup of $\pp^2$ in two different points. Up to a change of coordinates, we can suppose them to be $(1:0:0)$ and $(0:1:0)$. Assuming that the fibers of $\pi_1$ over these points are of positive dimension, we obtain the conditions $a_{0,1}=a_{0,2}=a_{1,0}=a_{1,2}=0$, and therefore the matrix has the form
$$
\left(
\begin{array}{ccc}
x_0 & x_1 & x_2\\
a_{0,0} x_0 + a_{2,0} x_2 &  a_{1,1} x_1 + a_{2,1} x_2 & a_{2,2} x_2
\end{array}
\right),
$$
which is equivalent, by taking a linear combination of the two rows, to
$$
\left(
\begin{array}{ccc}
x_0 & x_1 & x_2\\
a_{0,0} x_0 + a_{2,0} x_2 &  a_{1,1} x_1 + a_{2,1} x_2 & 0
\end{array}
\right).
$$
The quadrics defined by the order 2 minors of the previous matrix are given by
$$
\begin{array}{l}
\mathcal{Q}_1:=(a_{1,1}-a_{0,0})x_0x_1 + a_{2,1}x_0x_2 - a_{2,0}x_1x_2 = 0\vspace{1mm}\\
\mathcal{Q}_2:=(a_{0,0}x_0 + a_{2,0}x_2)x_2 = 0\vspace{1mm}\\
\mathcal{Q}_3:=(a_{1,1}x_1 + a_{2,1}x_2)x_2 = 0
\end{array}
$$
The intersection of the last two quadrics gives us the line $x_2=0$ and the point $P=(-a_{2,0}a_{1,1}:-a_{2,1}a_{0,0}:a_{0,0}a_{1,1})$.\\
Observe that either $\{x_2=0\} \subset \mathcal{Q}_1$ (which is not compatible with our hypothesis), or the intersection $\mathcal{Q}_1 \cap \{x_2=0\}$ gives us the two points $(1:0:0)$ and $(0:1:0)$. On the other hand $P \in \mathcal{Q}_1$ by direct computations. Since $S_{(1,1)}$ contains only two $1$--dimensional fibers of $\pi_1$, we conclude that $P$ is either $(1:0:0)$ or $(0:1:0)$. 
Let us describe the first case, the second being completely analogous by a coordinate change. We get $a_{0,0}=0$, which implies furthermore that $a_{1,1}\neq 0 \neq a_{2,0}$ and finally, the matrix in (\ref{mat-det-points}) can be given by
$$
\left(
\begin{array}{ccc}
x_0 & x_1 & x_2\\
a_{2,0} x_2 & x_1 + a_{2,1} x_2 & 0
\end{array}
\right).
$$
The case of $S_{(1,1)}$ with an $A_2$-type singularity can be fully described in a similar way. Assume that the starting blowup is at the point $(1:0:0)$, impose a positive dimensional fiber for the projection over this point and, finally, require that the intersection of the three quadrics given by the three minors is supported only on $(1:0:0)$. Several possibilities arise, that are easy to describe explicitly. Nevertheless, this type of del Pezzo surface does not appear when studying instanton bundles and hence we leave the details to the interested reader. Finally, let us suppose $S_{(1,1)}=S_{(1,0)}\cup S_{(0,1)}$ to be reducible. This implies that the three linear forms appearing in the second row of (\ref{mat-det-points}) are, potentially after having added a multiple of the first row, proportional to each other. Namely, we can supposed that the matrix is of the form
$$
\left(
\begin{array}{ccc}
x_0 & x_1  & x_2\\
\alpha\ell(x_0,x_1,x_2) & \beta\ell(x_0,x_1,x_2)& \gamma\ell(x_0,x_1,x_2)
\end{array}
\right)
$$
where $\ell$ is a linear form and $\alpha,\beta,\gamma$ are scalars. We see that the point $(\alpha:\beta:\gamma)$ is the one we blow up to obtain the component $S_{(0,1)}$ of $S_{(1,1)}$ lying in $|\sO_F(h_2)|$ and the line determined by $\ell$ is the projection of the component $S_{(1,0)}$ of $S_{(1,1)}$ lying in $|\sO_F(h_1)|$. Moreover, from this description we see that $S_{(1,0)}\cap S_{(0,1)}$ is an irreducible conic (resp. a reducible conic) if and only if $(\alpha:\beta:\gamma)\not\in \ell$ (resp. $(\alpha:\beta:\gamma)\in \ell$).

\subsection{A global description}\label{sGlobalDescription}
We can associate to each surface $S_{(1,1)}$ a square matrix $A=(a_{i,j})_{0 \leq i,j \leq 2}$ appearing in \eqref{eq-defS6}. Observe that any other matrix of the form $A + \lambda I_3$, where $\lambda \in \mathbb{C}$ and $I_3$ is the identity matrix, represents the same surface, so we can represent any $S_{(1,1)}$ by a matrix having 0 as an eigenvalue.

Denoting by $[A]$ a point of the projective space $\pp\left(\mathcal{M}_{3\times 3}\right) \simeq \pp^8$ of $3\times 3$  matrices, we have the surjective map
$$
\mathcal{D} \stackrel{\mathrm{pr}}{\longrightarrow} \pp^7
\quad \mbox{with} \quad
\mathcal{D} = \left\{ [A] \in \pp^8 \: | \: \det (A) = 0\right\},
$$
where $\mathbb P^7=\mathbb{P}(H^0(\mathcal{O}_F(h)))$ stands for the parameter space of surfaces $S_{(1,1)}$. Let us give more details about it. The map $\mathrm{pr}$ is finite of degree 3 from the cubic hypersurface $\mathcal{D}$ and it can be seen also as the projection of $\mathcal{D}$ from the point of $\pp^8$ corresponding to the class of the identity matrix. Its branch locus is exactly the divisor of singular del Pezzo surfaces in $\pp^7$ and, therefore, the corresponding ramification locus is the intersection of $\mathcal{D}$ with the locus of matrices for which the discriminant of the characteristic polynomial vanishes. Nevertheless, the fiber is 0-dimensional for any surface $S$ and the number of points in the fiber equals the number of different eigenvalues. Let us denote by:
\begin{itemize}
    \item $\Lambda_{sm}$ the locus of smooth $S_{(1,1)}$ surfaces;
    \item $\Lambda_{A_i}$ the locus of irreducible surfaces with an $A_i$-type singularity, respectively, for $i=1,2$;
    \item $\Lambda_r$ the locus of reducible $S_{(1,1)}$ surfaces.
\end{itemize}
Consider $\mathcal{C} = \mathrm{pr}^{-1} \left( \pp^7 \backslash \left(\Lambda_{sm} \cup \Lambda_{A_1}\right)\right)$ and notice that
$$
\mathcal{C} = \ \left\{[A] \in \mathcal{D} \: | \: \rk(A)=1 \right\} \ \cup \ \left\{
\begin{matrix}
    [A] \in \mathcal{D} \: | \: \mbox{the discriminant of the 
    first derivative}\\
    \mbox{of the characteristic polynomial of $A$ vanishes}
\end{matrix}
\right\}.
$$
This is the union of two closed subscheme of $\sD$, hence $\sC$ is closed. In the following Lemma we gather the results obtained so far since they will be fundamental in Theorem \ref{tModuliHooft}.
\begin{lemma}\label{lOpenDelPezzo}
    The surfaces of type $S_{(1,1)}$ contained in the flag variety $F$ are parameterized by $\mathbb{P}^7=\mathbb{P}(\mathcal{O}_F(h))$ in such a way that $\Lambda_{sm} \cup \Lambda_{A_1}$ forms an open subset of it.
\end{lemma}

\section{Multiple rational curves on the Flag variety}\label{sMultipleCurves}

In this section we will describe the geometry of particular families of rational curves that will appear as zero loci of sections of instanton bundles.

We start by describing the curves $C$ in $F$ which project to a point or a line of $\pp^2$ via one of the natural projections. This property forces the class of $C$ to be $h_i^2+a h_j^2$ for some non-negative $a \in \mathbb{Z}$. Let us begin with the following proposition.
\begin{proposition}\label{com-int}
Let $C\subset F$ be a connected reduced curve of class $h_i^2+ah^2_j$ with $a \ge 1$. Then $C$ is a complete intersection of type $h_j$, $h_i+(a-1)h_j$ and arithmetic genus $p_a(C)=0$. Moreover, if $C$ is an integral curve, then $C\cong\pp^1$ and its normal sheaf is $\mathcal{N}_{C\mid F}\cong\sO_{\pp^1}(1)\oplus\sO_{\pp^1}(2a-1)$.
\end{proposition}
\begin{proof}
Let $C\subset F$ be a curve as in the hypothesis. Then $\pi_j(C)=L\subset\pp^2$ is a line and therefore $C\subset S$ for $S$ a surface in the linear system $|h_i|$. $S$ is a cubic scroll with canonical divisor $K_S=-2C_0-3f$. 

Let $C=cC_0+df$. By Lemma \ref{lRestrictionPic} we can express $C=\bigl(c(h_i-h_j)+dh_j\bigr)_{\mid S}=ch_ih_j+(d-c)h_j^2=ch_i^2+dh_j^2$. Therefore, $c=1$ and $d=a$. From the adjunction formula, it is immediate to check that any curve $C=C_0+af$ is of arithmetic genus $p_a(C)=0$ and  degree $a+1$. To conclude, consider the exact sequence:
\[
0\to\sO_F\bigr(h_i+(a-2)h_j\bigl)\to\sO_F\bigr(h_i+(a-1)h_j\bigl)\xrightarrow{\phi}\sO_S\otimes\sO_F\bigr(h_i+(a-2)h_j\bigl)\cong\sO_S(C)\to 0.
\]
Since $h^1(\sO_F(h_i+(a-2)h_j))=0$ for $a\geq 1$, the induced map $H^0(\phi)$ is surjective and therefore any curve in the linear system $|C|$ is a complete intersection. The statement about the normal bundle follows directly from the fact that $C$ is a complete intersection. Notice also that $\omega_C\cong\sO_C(-2h_j)$ in the integral case.
\end{proof}

In order to deal with multiple structures arising on such curves it is useful to describe the ideal of the first infinitesimal neighbourhood $C^{(1)}$ of a curve $C$.

\begin{lemma}\label{c2}
Let $C\subset F$ be as in Proposition \ref{com-int}. Then $C^{(1)}$ has the following $\sO_F$-resolution:
$$
\begin{array}{ccccccccc}
 & & & \sO_F(-2h_i)\\
     &  \sO_F(-h_j-(a+1)h_i) & & \oplus \\
 0 \arr & \oplus & \stackrel{M}\arr & \sO_F(-2h_j-(2a-2)h_i) & \arr & \sO_{F} &\arr \sO_{C^{(1)}} \arr 0  \\    
     &  \sO_F(-2h_j-(2a-1)h_i) & & \oplus  \\
 & & & \sO_F(-h_j-ah_i)
\end{array}
$$
where $M$ can be represented by the matrix
$$
\left(
\begin{array}{cc}
     \zeta & 0  \\
     0 & \vartheta \\
     -\vartheta & \zeta
\end{array}
\right)
$$
in which $\vartheta \in H^0(\sO_F(h_j))$ and $\zeta \in H^0(\sO_F(h_i+(a-1)h_j))$ are the two generators of $I_C$. In particular, $\chi(\sO_{C^{(1)}})=3-2a$.
\end{lemma}
\begin{proof}
The statement follows directly from the fact that $I_{C^{(1)}} = \langle \vartheta^2, \zeta^2, \vartheta\zeta\rangle$ is a standard determinantal ideal defined by the maximal minors of the matrix representing $M$.
\end{proof}

\begin{remark}\label{rChiThick}
We have similar (and simpler) statements when $L\subset F$ is a curve having class $h_i^2$. Namely $L$ is a line and complete intersection of type $h_i,h_i$. The normal sheaf is $\mathcal{N}_{L\mid F}\cong\sO_{\pp^1}^2$.
Moreover, its first infinitesimal neighbourhood $L^{(1)}$ has $\sO_F$-resolution

$$
0\arr \sO_F(-3h_i)^2\arr\sO_F(-2h_i)^3\arr \sO_F\arr\sO_{L^{(1)}}\arr 0
$$
and $\chi(\sO_{L^{(1)}})=3$.
\end{remark}

We will now deal with multiple structures supported on curves representing a class $h_i^2+ah_j^2$ with $a\ge 0$ in the Chow ring of $F$.

\subsection{Quasi-primitive extensions}

Let $C$ be a rational, smooth, complete intersection curve described in Proposition \ref{com-int}. We are now interested in the structure of non-reduced curves with support $C$. We postpone the study of  multiple structures supported on a line to the next subsection.

We can now start to describe Cohen-Macaulay double structures on the curves $C$, which are all obtained by the Ferrand doubling technique. We will use \cite{BF,Ma} as standard references for this section. Let us denote by $\nu_X$ the conormal bundle of the variety $X$. For a smooth rational curve $C$ on $F$, by Proposition \ref{com-int} we have $\nu_C \cong \sO_{\pp^1}(-1) \oplus \sO_{\pp^1}(1-2a)$. Every Ferrand double structure $Y_1$ on $C$ arises from a surjective morphism
\[
\nu_C \xrightarrow{\phi} \mathcal{L} \to 0
\]
where $\sL$ is a line bundle on $C$, thus $\sL \cong \sO_{\pp^1}(\alpha)$. Notice that since $\phi$ is surjective, we have $\alpha \ge -1$. In particular we have the following short exact sequence
\begin{equation}\label{eSeqNilpotent}
0 \rightarrow \frac{\sI_{Y_1}}{\sI_C^2} \rightarrow \frac{\sI_{C}}{\sI_C^2} \xrightarrow{\phi} \sL\cong \frac{\sI_C}{\sI_{Y_1}} \rightarrow 0.
\end{equation}

In order to study higher multiplicity extensions, let us start by focusing on the Cohen-Macaulay extensions $Y$ which are locally contained in a smooth surface. These are the so-called \textit{primitive extensions} of $C$, according to the following definition.

\begin{definition}
Let $C$ be a smooth integral curve. A \textit{primitive extension} of $C$ is a Cohen-Macaulay curve $Y$ such that $Y_{red}\cong C$ and such that $Y$ can be locally embedded in a smooth surface. Associated to $Y$ there is a canonical filtration 
$$
C=Y_0\subset Y_1\subset\dots Y_k=Y 
$$
where $Y_j = Y \cap C^{(j)}$ and $C^{(j)}$ is the $j$-th infinitesimal neighbourhood of $C$. The integer $k+1$ is the multiplicity of $Y$. In this situation, $\sL:=\sI_{C\mid Y_1}$ is a line bundle on $C$. It is called \textit{the type} of $Y$.
\end{definition}

Let us describe primitive extensions of multiplicity $k+1$ of type $\sL$. For $j=1,\ldots,k$, we have exact sequences
\begin{equation}\label{strSheaf}
0 \rightarrow \sL ^j \rightarrow \sO_{Y_j} \rightarrow \sO_{Y_{j-1}} \rightarrow 0.
\end{equation}
Moreover we have the exact sequence
$$
0 \rightarrow \sL^{k}\to \frac{\sI_{Y}}{\sI_C \sI_Y} \rightarrow \frac{\sI_{C}}{\sI_C^2} \rightarrow \sL \rightarrow 0
$$
and in particular $\omega_{Y|C}\cong\omega_C \otimes \sL^{-k}$. Thus in order to effectively compute the canonical sheaf of $Y$, it is essential to understand the behaviour of the restriction map $\Pic(Y) \to \Pic(C)$. The following paragraphs deal with these issues for primitive extensions of rational curves $C$ of type $\sO_C$ which will be related to instanton bundles.

\begin{lemma}
\label{lPicMultiple}
Let $C\subset F$ be a rational curve and let $Y$ be a primitive extension of $C$ of type $\sO_C$. Then the restriction map $\Pic(Y) \to \Pic(C)$ is an isomorphism.
\end{lemma}
\begin{proof}
Since $\sO_C\cong \sI_{Y_{j-1}}/\sI_{Y_j}$ is an ideal of square zero in $\sO_{Y_j}$,  sequence \eqref{strSheaf} yields the short exact sequences
$$
0 \rightarrow \sO_C \rightarrow \sO_{Y_j}^\ast \rightarrow \sO_{Y_{j-1}}^\ast \rightarrow 0,
$$
for each $j \le k$. Hence we get the exact sequence in cohomology
$$
H^1(\sO_C) \to \Pic(Y_j) \xrightarrow{\phi_j} \Pic(Y_{j-1}) \to H^2(\sO_C).
$$
Since $H^1(\sO_C)\cong H^2(\sO_C)\cong 0$, restriction map $\phi_j$ is an isomorphism for each $j$.
\end{proof}
As a straightforward consequence of the above lemma, we infer that  the restriction $\omega_{Y|C}$ completely determines the canonical sheaf $\omega_Y$ in the case of primitive extensions of type $\sO_C$.
 
Now we will explicitly describe the normal bundle of a primitive extension $Y$ of type $\sO_C$ and multiplicity $k+1$ with support a rational smooth curve described in Proposition \ref{com-int}. In order to do so, following the notation introduced before, we recall the following two short exact sequences -we will simply use the notation $\sI_Y$ when considering the inclusion $Y\subset F$-:  
\begin{gather}\label{eIdealMult}
0 \to  \frac{\sI_C^{k+1}}{\sI_{Y_1}\sI_{C}^k} \to \frac{\sI_Y}{\sI_C \sI_Y} \to \frac{\sI_{Y_1}}{\sI_C^2} \to 0,  \notag \\ 
\\ 
0 \to \frac{\sI_{Y_1}}{\sI_C^2} \to \frac{\sI_C}{\sI_C^2}\to \frac{\sI_C}{\sI_{Y_1}} \to 0. \notag
\end{gather}
Changing the entry of the second short exact sequence according to known isomorphisms, we get 
\begin{equation}\label{mult2}
0 \rightarrow \frac{\sI_{Y_1}}{\sI_C^2} \rightarrow \sO_{\pp^1}(-1)\oplus \sO_{\pp^1}(1-2a) \rightarrow \sO_C \rightarrow 0.
\end{equation}

This implies that
$$
\frac{\sI_{Y_1}}{\sI_C^2} \simeq \sO_{\pp^1}(-2a)
$$
and furthermore
$$
\frac{\sI_Y}{\sI_C \sI_Y} \simeq \frac{\sI_{Y_1}}{\sI_C^2} \oplus \frac{\sI_C^{k+1}}{\sI_{Y_1}\sI_{C}^k} \simeq \sO_{\pp^1}(-2a) \oplus \sO_{\pp^1}.
$$

This means that the restriction of the conormal bundle $\sN^{\vee}_{Y}$ to the curve $C$ is isomorphic to $\sO_{\pp^1}(-2a) \oplus \sO_{\pp^1}$, or, equivalently,
\begin{equation}
\label{eDetNormalMultiple}
\sN_{Y|F}  \otimes \sO_C\cong \sO_{\pp^1} \oplus \sO_{\pp^1}(2a).
\end{equation}
In order to determine $\sN_{Y|F}$, we can, analogously to Lemma 1.2 from \cite{BeFr},  strengthen Lemma \ref{lPicMultiple}  by means of the following Proposition. 

\begin{proposition}\label{pSplitMultiple}
Let $C$ be a smooth curve satisfying the hypotheses of Proposition \ref{com-int} and let $Y$ be a  primitive extension of multiplicity $k+1$ and type $\sO_C$. Then any locally free sheaf on $Y$ splits.
\end{proposition}
\begin{proof}
Let $\sG$ be a locally free sheaf supported on $Y$. Since $\Pic (Y) \cong \Pic(C)$ thanks to Lemma \ref{lPicMultiple}, there exists a minimal integer $t$ such that $h^0(\sG(t))>0$. Let $s \in H^0(\sG(t))$ be a section. We claim that $s$ has no zeros. Indeed, suppose $s$ vanishes at a point $y \in C\subset Y$. Let us recall that, given the smooth integral curve $C$ and a line bundle $\mathcal{L}$ on it, for any $n$ there always exists a primitive multiple structure on $C$ of multiplicity $n$ admitting a retraction $\pi:Y \to C$ (namely, a map $\pi$ such that the composition with the inclusion $C\subset Y$ is the identity). It is constructed as a section in the total space $\mathrm{Spec}(\mathcal{L})$. But, when $g(C)=0$ and $\deg(\mathcal{L})\ge 0$, the primitive multiple structure over $C$ of multiplicity $n$ and type $\mathcal{L}$ is unique, so it should be the one with a retraction (cf. \cite[5.2.8]{Dre}). In particular, our primitive curve $Y$ of type $\sO_C$ has a retraction $\pi:Y \to C$. The map $\pi$ is flat and the fibers of $\pi$ are curvilinear multiple points, i.e. zero-dimensional schemes isomorphic to $\mathrm{Spec}( \mathbb{C}[x]/x^{k+1})$. Thus $Y$ is isomorphic as a scheme to $\mathbb{P}^1 \times \mathrm{Spec}( \mathbb{C}[x]/x^{k+1})$ and it follows that if $s$ vanishes at the simple point $y$, it vanishes along the entire fiber $\pi^{-1}(y)$. This fiber is a divisor $D$ in $Y$ whose ideal is $\sI_{D|Y}\cong \sO_Y(-1)$. However  in this case $s$ would produce a section of $\sG(t-1)$ contradicting the minimality of $t$. The rest of the proof follows verbatim the proof of Grothendieck's Theorem \cite[Theorem 2.1.1]{O-S-S}.
\end{proof}
As a direct consequence of Lemma \ref{lPicMultiple} and Proposition \ref{pSplitMultiple} it follows that 
\begin{equation}
\label{eNormalMultiple}
\sN_{Y|F} \cong \sO_Y \oplus \sO_Y(2h_i).
\end{equation}

When $C$ is a smooth conic, we can explicitly construct the ideal of $Y$ by describing in more detail the exact triple \eqref{mult2}. Let us start describing the ideal of $Y_1$. Suppose that $I_C = \langle x_0,y_0 \rangle$. The epimorphism in the exact sequence is represented by two linearly independent forms in $H^0(\sO_C(1))$. As maps of $\sO_F$-modules, they are given by $\vartheta = \lambda_1 y_1 + \lambda_2 y_2$ and $\zeta = \mu_1 x_1 +\mu_2 x_2$. Since $x_0$ and $y_0$ are generators of  $\frac{\sI_C}{\sI_C^2}$, we find that $\frac{\sI_{Y_1}}{\sI_C^2}$ is generated by $x_0 \vartheta + y_0 \zeta$ and therefore $I_{Y_1} = \langle x_0^2, x_0y_0, y_0^2, x_0 \vartheta + y_0 \zeta \rangle$. Observe that a del Pezzo surface $S_{(1,1)}$, which contains the primitive extension, is defined by $x_0 \vartheta + y_0 \zeta + \alpha x_0 y_0$. 

\begin{remark}\label{uniqueDP}
From the above  representation of the ideal $I_{Y_1}$ we see that for a smooth conic $C$ in $F$, a double extension of C of type $\sO_C$ is contained in a $\pp^1$ of del Pezzo surfaces. Nevertheless, once we fix a line in $F$, this pencil intersects the set of del Pezzo surfaces containing the line only at one point. This explains why, having fixed a double conic and a line, we have a unique $S_{(1,1)}$ containing them. 
\end{remark}

In general, let $Y$ and $\tilde{Y}$ be two primitive extensions of type $\sO_C$ and multiplicity $k+1$ and $k+2$ respectively, supported on $C$ and such that $\tilde{Y} \supset Y$. Then
\begin{equation}\label{iteration-mult}
0 \rightarrow \frac{\sI_{\tilde{Y}}}{\sI_Y\sI_C} \rightarrow \frac{\sI_{Y}}{\sI_Y\sI_C} \rightarrow \sO_C\rightarrow 0.
\end{equation}
  Our goal is to construct the ideal $I_{\tilde{Y}}$ starting from the ideal $I_Y$. We can  rewrite (\ref{iteration-mult}) as
$$
0 \rightarrow \frac{\sI_{\tilde{Y}}}{\sI_Y\sI_C} \rightarrow \underbrace{\frac{\sI_{Y_1}}{\sI_C^2} \oplus \frac{\sI_C^{k+1}}{\sI_{Y_1}\sI_{C}^k}}_{\sO_{\pp^1}(-2) \oplus \sO_\pp^1}  \rightarrow \sO_C \rightarrow 0.
$$
Iterating, we obtain:
\begin{equation}
\label{IdealMultConic}
I_{\tilde{Y}} = \langle x_0 \vartheta + y_0 \zeta + \alpha x_0y_0, (x_0,y_0)^{k+2}\rangle.
\end{equation}
Finally, we can specify what are the admissible values for $\alpha$. Indeed, by \cite{ABBS} we know that, in order for the del Pezzo sextic to be smooth and irreducible, the matrix $A$ defining the del Pezzo surface as in \eqref{eq-defS6} must have three different eigenvalues. This is equivalent to requiring that $\alpha^2 \neq 4(\lambda_1 \mu_1 + \lambda_2 \mu_2)$.

We now deal with quasi-primitive extensions of multiplicity $k+1\ge 2$.

\begin{definition}\label{dTrivialThick} A multiple structure $Y$ on a smooth integral curve $C$ is called {\em quasi-primitive} if $Y$ is a Cohen-Macaulay curve such that $Y$ does not contain the first infinitesimal neighbourhood $C^{(1)}$ of $C$.  Otherwise, if $Y$ does contain $C^{(1)}$, $Y$ is a {\em thick extension} (cf. \cite[Section 3.4]{BF}). \end{definition}

Let us explicitly describe a multiplicity $k+1$, quasi--primitive extension $Y$ of a rational curve $C$ as in Proposition \ref{com-int}. Consider the filtration
\[
C=Y_0 \subset Y_1 \subset \dots \subset Y_k=Y
\]
where $Y_i=Y \cap C^{(i)}$. Let us denote $\sI_j:=\sI_{Y_j}$. By \cite[Section 3]{BF},  $\sI_{j-1}/\sI_{j}$ are line bundles on $C$; let us denote them by $\sL_j$. Furthermore, the maps
\[
\sL^{\otimes j} \to \sL_j
\]
are generically surjective, thus $\sL_j \cong \sL^{\otimes j} \otimes \sO_C(D_j)$ for some effective divisors $D_j$. Moreover we have the short exact sequence
\begin{equation}\label{strSheaf2}
0 \rightarrow \sL_j \rightarrow \sO_{Y_j} \rightarrow \sO_{Y_{j-1}} \rightarrow 0,
\end{equation}
which yields
\begin{equation}\label{eChiMultiple}
\chi(\sO_Y) = \chi (\sO_C) + \sum_{j=1}^{k}{\chi(\sL_j)}.
\end{equation}
Since $\sL\cong \sO_{\pp^1}(\alpha)$ and $\sO_C(D_j)\cong \sO_{\pp^1}(d_i)$ for some $\alpha\geq -1$ and $d_i \ge0$, the equation \eqref{eChiMultiple} becomes
\begin{equation}\label{eGenusMultiple}
p_a(Y) = -\sum_{j=1}^{k}{(\alpha+1+d_i)}.
\end{equation}
Notice that a multiplicity $k+1$, quasi-primitive extension $Y$ is primitive of type $\sO_{\pp^1}(\alpha)$ if and only if $d_i=0$ for all $i$, and in this case $p_a(Y) = -k- \frac{k(k+1)}{2}\alpha$.

\subsection{Multiple structures on lines}\label{sNonPrimitiveMultiple}

We will now deal with  Cohen-Macaulay multiple curves $Y$ whose reduced structure $Y_{red}=L$ is a line in $F$ from the class $h_i^2$ and satisfying particular vanishing conditions in cohomology. These requirements will appear when studying the zero locus of sections of instanton bundles. 

Specifically, let us consider Cohen-Macaulay one dimensional schemes $Y$, supported on the line $L$ in the family $|h_1^2|$, such that $h^0(\sO_Y(-h_2)) = h^1(\sO_Y(-h_2))= 0$ (the case $L \in |h_2^2|$ being completely symmetric).

Define, as explained in \cite{BF}, $\mathcal{J}_i$ as the ideal associated to the largest Cohen-Macaulay subspace $Y_i \subset Y\cap L^{(i-1)}$, hence 
$$
\mathcal{J}_i \supset \sI_Y + \sI_{L}^i.
$$
We know that $\mathcal{J}_i / \mathcal{J}_{i+1}$ is a locally free $\sO_L$-module and we will set the notation
$$
E_i := \frac{\mathcal{J}_{i-1}}{\mathcal{J}_{i}} \simeq \bigoplus_r \sO_L(\beta^i_r).
$$
Notice that if $E_i$ is a line bundle for all $i$, then $Y$ is a quasi-primitive extension on $L$. Furthermore, we have a generically surjective map
\begin{equation}\label{gensur-map}
    E_1^{\otimes i } \longrightarrow E_i.
\end{equation}
If we consider the short exact sequence
$$
0 \longrightarrow \bigoplus_r \sO_L(\beta^t_r) \longrightarrow \sO_Y \longrightarrow \sO_{Y_{t-1}} \longrightarrow 0,
$$
the vanishing $H^0(\sO_Y(-h_2)) = 0$ implies that $\beta^t_r\leq 0$, for any $r$.

Consider now the first extension
\begin{equation}\label{seq-firststep}
0 \longrightarrow \bigoplus_r \sO_L(\beta^1_r) \longrightarrow \sO_{Y_2} \longrightarrow \sO_{L} \longrightarrow 0.
\end{equation}
As $H^1(E_1(-h_2)) \simeq H^1(\sO_{Y_2}(-h_2))$ and due to the surjective maps, for any $i>0$,
$$
H^1(\sO_{Y_{i+1}}(-h_2)) \longrightarrow H^1(\sO_{Y_i}(-h_2)),
$$
we have $H^1(E_1(-h_2)) = 0$. Indeed, if that is not the case, this would imply $H^1(\sO_Y(-h_2)) \neq 0$, a contradiction. Hence, $\beta^1_r \geq 0$.

Suppose that $\beta_r^t < 0$ for at least one value of $r$. The previous inequalities give then a contradiction with the map described in \eqref{gensur-map}. Therefore, $\beta_r^t = 0$ for any $r$. Applying the same technique iteratively, we obtain the $\beta_r^j = 0$ for any $j\geq 2$. Finally, Sequence (\ref{seq-firststep}) implies that 
$\beta_r^1 = 0$ for any $r$ as well, hence we have the following short exact sequences
\begin{equation}\label{eStrThick}
0 \longrightarrow \sO_L^{\oplus r_j} \longrightarrow \sO_{Y_j} \longrightarrow \sO_{Y_{j-1}} \longrightarrow 0.
\end{equation}

This shows that $Y$ is a specific multiple structure, namely, it is of type $\sO_L$.
If $Y$ is not thick (see Definition \ref{dTrivialThick}), it is a primitive extension.

Since in this case $Y$ can be obtained by iterative extensions of direct sums of $\sO_L$, we find the resolution of its structural sheaf. Thanks to the horseshoe lemma, we construct the following commutative diagram, that gives the resolution of a sheaf $\sF \in \Ext^1(\sO_L^{\oplus \beta},\sO_L^{\oplus \alpha})$
$$
\xymatrix{
 & & & & 0 \ar[d] &\\
 0 \ar[r] & \sO_F(-2h_1)^{\oplus \alpha} \ar[r] & \sO_F(-h_1)^{\oplus 2\alpha} \ar[r] & \sO_F^{\oplus \alpha} \ar[r]  & \sO_L^{\oplus \alpha} \ar[r] \ar[d] & 0 \\
 0 \ar[r] & \sO_F(-2h_1)^{\oplus \alpha+\beta} \ar[r] & \sO_F(-h_1)^{\oplus 2\alpha+2\beta} \ar[r] & \sO_F^{\oplus \alpha+\beta} \ar[r]  & \sF \ar[r] \ar[d] & 0\\
 0 \ar[r] & \sO_F(-2h_1)^{\oplus \beta} \ar[r] & \sO_F(-h_1)^{\oplus 2\beta} \ar[r] & \sO_F^{\oplus \beta} \ar[r] & \sO_L^{\oplus \beta} \ar[r] \ar[d]& 0\\
 & & & & 0  &
}
$$
Indeed, we can apply the horseshoe lemma because the second column from right is equivalent to the long exact sequence in cohomology of the sheaves in the rightmost column. Applying iteratively the previous diagram, we have
$$
0 \longrightarrow \sO_F(-2h_1)^{\oplus k} \longrightarrow \sO_F(-h_1)^{\oplus 2 k} \longrightarrow \sO_F^{\oplus k} \longrightarrow \sO_Y \longrightarrow 0,
$$
where $k = 1 + \sum_i \rk(E_i)$ denotes the multiplicity of $Y$.

This implies that $Y = \pi_1^{-1}(Z)\cong Z\times \pp^1$, where $\pi_1: F \rightarrow \pp^2$ is the projection on the first projective plane and $Z\subset\pp^2$ a 0-dimensional scheme, supported on the simple point $P = \pi_1(L)$.

We conclude this section by showing that all multiple structures given as in Definition \ref{dTrivialThick} are complete intersections, which allows us to describe also their normal bundle. 

\begin{lemma}\label{lThickComInt}
Let $Y$ be a multiple structure on a line $L$ of type $\sO_L$. If $Y$ is a locally complete intersection, then it is a global complete intersection. Moreover, in this case the normal bundle $\sN_Y$ is given by
\begin{equation}\label{eNormalThick}
\sN_Y \cong \sO_Y^{\oplus 2},
\end{equation}
and the restriction map $\Pic(Y) \to \Pic(L)$ is an isomorphism.
\end{lemma}
\begin{proof}
Let $Z$ be the projection of $Y$ via $\pi_1$. Since $Y = \pi_1^{-1}(Z)\cong Z\times \pp^1$, all the local rings $\sO_{Y,q}$, with $q\in L$, are isomorphic to $\sO_{Z,p}$. If $Y$ is a locally complete intersection, then $Z$, and therefore $Y$ itself, are global complete intersections. The statement on the normal bundle then follows directly. Finally, arguing as in Lemma \ref{lPicMultiple}, we get the isomorphism $\Pic(Y) \cong \Pic(L)$.
\end{proof}

\section{$h_i$-'t Hooft instantons}\label{sHitHooft}
In this section, once we have recalled the definition of instanton bundles (cf. \cite{MMP} for more details), we introduce the notion of 't Hooft bundles on the flag variety.

\begin{definition}\label{dInsta}
A rank  two vector bundle $\sE$ on $F$ is an \textit{instanton bundle} of charge $k$ if the following properties hold:
\begin{itemize}
\item $c_1(\sE)=0$;
\item $c_2(\sE)=kh_1h_2$;
\item $h^1(\sE(-h))=0$ (the so-called "instantonic condition");
\item $h^0(\sE)=0$ and $\sE$ is $\mu$-semistable with respect to $h=h_1+h_2$.
\end{itemize}
Furthermore given an effective divisor $D$ on $F$, an instanton bundle $\sE$ is a $D$\textit{-'t Hooft bundle} if and only if $h^0(\sE(D)) \neq 0$.
\end{definition}

\begin{remark}
The charge of an instanton bundle is bounded from below. Indeed \eqref{eulerchar} yields $h^1(\sE)=2k-2$, thus $k\ge 1$. Instanton bundles of minimal charge (i.e. $k=1$) are Ulrich bundles, according to the following definition (see \cite{CMRP} for more details on Ulrich bundles).
\end{remark}

\begin{definition}
    Let $(X,\sO_X(h))$ be a smooth polarized projective variety. A vector bundle $\sE$ on $X$ is called \textit{arithmetically Cohen-Macaulay} if $H^i(\sE(th))=0$ for $0< i < \dim(X)$ and any $t\in \mathbb{Z}$. A vector bundle $\sE$ is called \textit{Ulrich} if it is arithmetically Cohen Macaulay and 
    $$0=h^0(\sE(-h))<h^0(\sE)=\deg(X)\rk(\sE).$$
\end{definition}

Our first goal is, given an instanton bundle $\sE$, to explicitly describe the zero locus of an element of $H^0(\sE(h_i))$. Let $s_i$ be a section of $\sE(h_i)$, then we have the following short exact sequence
\begin{equation}\label{eSerreHooft}
0 \to \sO_F \xrightarrow{s_i} \sE(h_i) \to \sI_{Y|F}(2h_i) \to 0.    
\end{equation}
Since $\sE$ has no global section, $Y$ is a purely two-codimensional subscheme in $F$. By the adjunction formula
\[
\omega_Y \cong \omega_F \otimes \det(N_{Y|F}),
\]
thus if $\sE$ is locally free we get $\omega_Y \cong \sO_F(-2h_j) \otimes \sO_Y$ with $i \neq j$. If $\sE$ is an instanton, $h^p(\sE(-h))=0$ for all $p$, thus also $\sI_{Y|F}(-h_j)$ is acyclic. Tensoring the standard short exact sequence
\begin{equation}\label{eStandardIdeal}
0 \to \sI_{Y|F} \to \sO_F \to \sO_Y \to 0
\end{equation}
by $\sO_F(-h_j)$ we see that $h^0(\sO_Y(-h_j))=h^1(\sO_Y(-h_j))=0$. Our next goal is to characterize such curves.

\begin{remark}\label{rTorsionFree}
If we consider any curve $Y$ such that $h^0(\sO_Y(-h_j))=h^1(\sO_Y(-h_j))=0$, then through the Serre's correspondence (c.f. \cite[Theorem 1]{Ar}) we obtain a (non-necessarily locally free) torsion free sheaf satisfying all the cohomological vanishings of Definition \ref{dInsta}.

\end{remark}

\begin{lemma}\label{lCanonicalMultiple}
 Let $C$  be the connected union of a smooth rational curve $Z$ representing $h_i^2+ah_j^2$ and a line $L$ representing $h_j^2$.
If $Y \subset F$ is a multiple structure supported on the curve $C$, then $\omega_Y \not\simeq \sO_Y(-2h_j)$.
\end{lemma}

\begin{proof}
The curve $C$ satisfies $h^0(\sO_C(-h_j))=h^1(\sO_C(-h_j))=0$, however $\omega_C \cong \sO_C(-h_i+(a-1)h_j) \not\cong \sO_C(-2h_j)$. Now we want to show that 
for any multiple structure $Y$ on such curve $C$, we also have $\omega_Y\not\cong\sO_Y(-2h_j)$. Suppose by contradiction that $\omega_Y\cong\sO_Y(-2h_j)$ and consider a non trivial extension
\[
0 \to \sO_F(-h_i) \to \sF \xrightarrow{\phi} \sI_{C|F}(h_i)\to 0,
\]
which is possible because $\Ext^1(\sI_{C|F}(h_i),\sO_F(-h_i)) \simeq H^1(\sO_C(-2h_j)) \neq 0$. Since $\omega_C\not\cong \sO_C(-2h_j)$ along $L$, by Serre's correspondence $\sF$ is not a vector bundle; indeed $\mathrm{Sing}(\sF)=L$. Since $C$ has pure dimension 1, we have 
$$
\sE xt^i(\sF,\sO_F) \cong \sE xt^i(\sI_{C|F}(h_i),\sO_F) = 0 \mbox{ for } i=2,3.
$$
This implies that, necessarily, the support of $\sE xt^1(\sF,\sO_F)$ is exactly $L$. We claim that $\sF^{\lor\lor} \not\cong \sO_F(-h_i) \oplus \sO_F(h_i)$. Suppose we do have this isomorphism. Since $\sF$ is torsion free, it injects in its double dual, thus we have the following commutative diagram
$$
\xymatrix{
 & 0 \ar[d] & 0 \ar[d] & 0 \ar[d] \\
 0 \ar[r] & \sO_F(-h_i) \ar[r] \ar[d] & \sF \ar[r] \ar[d] & \sI_{C|F}(h_i) \ar[r] \ar[d] & 0\\
 0 \ar[r] & \sO_F(-h_i) \ar[r]^-{\alpha} & \sO_F(-h_i) \oplus \sO_F(h_i) \ar[r] \ar[d] & \sH \ar[r] \ar[d] & 0\\
 & & \sQ \ar[r] \ar[d] & \sQ \ar[d]\\
 & & 0 & 0
}
$$
Notice that $\alpha$ is either defined as $\alpha=(1 \:\:|\:\: \alpha_2)$ or $\alpha=(0 \:\:|\:\: \alpha_2)$. In the first case $\sF \cong \sO_F(-h_i) \oplus \sI_{C|F}(h_i)$, thus a contradiction. In the second case, we have $\sH \cong \sO_F(-h_i) \oplus \sO_{S}(h_i)$, with
$$
0 \longrightarrow \sO_F(-2h_i) \longrightarrow \sO_F \longrightarrow \sO_S \longrightarrow 0.
$$
 However $\Hom(\sI_{C|F}(h_i),\sO_F(-h_i)) = 0$ and there is no injective morphism $\sI_{C|F}(h_i) \hookrightarrow \sO_{S}(h_i)$, so $\sF^{\vee\vee} \not \cong \sO_F(-h_i)\oplus \sO_F(h_i).$

Let us continue considering the short exact sequence 
\[
0 \to \sI_{Y|F}(h_i) \to \sI_{C|F}(h_i) \xrightarrow{\psi} \sI_{C|Y}(h_i) \to 0.
\]
From it, the surjective composition $\psi\phi$ yields
\[
0\to \sO_F(-h_i) \to \sG \to \sI_{Y|F}(h_i) \to 0.
\]
where $\sG:=\ker(\psi \phi)$. Let us prove that $\sG$ is not the trivial extension. Otherwise from the short exact sequence
\begin{equation}\label{FtoG}
    0\longrightarrow \sG \longrightarrow \sF \xrightarrow{\psi \phi} \sI_{C|Y}(h_i) \longrightarrow 0,
\end{equation}
we obtain, recalling that $\sE xt^i(\sI_{C|Y}(h_i),\sO_F)=0$ for $i=0,1$ by \cite[III 7.3]{Har}, that
$$
\sF^\lor \simeq \sG^{\lor} \simeq \sO_F(-h_i) \oplus \sO_F(h_i)
$$
which does not hold.

As $\omega_Y\cong\sO_Y(-2h_j)$ and $\sG$ is a non-trivial extension, $\sG$ is a vector bundle by means of the Serre's correspondence. In particular, $\sE xt^1(\sG, \sO_F) = 0$. Applying $\sH om(-,\sO_F)$ to Sequence (\ref{FtoG}), we have an inclusion
$$
\sE xt^1(\sF, \sO_F) \hookrightarrow \sE xt^1(\sG, \sO_F),
$$
leading to contradiction. Therefore the canonical sheaf of $Y$ cannot have the considered form, proving our result.

\end{proof}

\begin{theorem}\label{theor-equiv}
Let $Y\subset F$ be a locally complete intersection curve. The following are equivalent:
\begin{enumerate}
\item $\omega_Y\cong\sO_Y(-2h_j)$ and $h^0(\sO_Y(-h_j))=0$.
\item $Y$ is the disjoint union of curves of one of these two types: 
\begin{itemize}
\item primitive extensions of type $\sO_C$ on smooth rational curves $C$ of class $h_i^2+ah_j^2$ with $a\ge 1$;
\item complete intersection multiple structures of type $\sO_C$ on lines $C$ of class $h_i^2$.
\end{itemize}
\end{enumerate}
\end{theorem}

\begin{proof}
$(1)\Rightarrow(2):$ 

Since the same conditions hold for any connected component of $Y$, we can suppose $Y$ to be a connected curve. Let $C:=Y_{red}$ and let us consider the short exact sequence
\begin{equation}\label{eIdealReduced}
0 \to \sI_{C|Y} \to \sO_Y \to \sO_C\to 0
\end{equation}
Apply the contravariant functor $\mathcal{H}om(-,\omega_F)$ to \eqref{eIdealReduced}. We obtain
\[
\sE xt^1(\sI_{C|Y},\omega_F) \to \omega_C\to \omega_Y \to \sE xt^2(\sI_{C|Y},\omega_F).
\]
By \cite[III 7.3]{Har},we have $\sE xt^1(\sI_{C|Y},\omega_F)=0$, thus there is an injective map
\begin{equation}\label{eInjCanonical}
0 \to \omega_C\to \omega_Y.
\end{equation}
Since $h^0(\omega_Y(h_j))=0$ by hypothesis, the same holds for $C$, i.e., $h^0(\omega_C(h_j))=0$. 
Suppose now that $C$ is reducible and write $C=C_1 \cup C_2$, with $C_1$ irreducible. From the short exact sequence
\begin{equation}\label{eMV}
0 \to \omega_{C_1} \oplus \omega_{C_2} \to \omega_{C} \to \omega_{C_1 \cap C_2} \to 0
\end{equation}
we also have $h^0(\omega_{C_1}(h_j))=0$, thus $h^0(\omega_{C_1})=0$. As $C_1$ is integral, we have $h^0(\sO_{C_1})=1$ and, in particular, 
$$
p_a(C_1)=1-\chi(\sO_{C_1})=1-h^0(\sO_{C_1})+h^0(\omega_{C_1})=0,
$$
i.e., $C_1$ is a smooth rational curve. In order to compute its class, suppose that $C_1$ has class $bh_i^2+ah_j^2$ in $A^2(F)$. By Riemann-Roch we have
$$
0\geq \chi(\omega_{C_1}(h_j))=2p_a(C_1)-2+\deg_{\sO_F(h_j)}(C_1)+1=-1+b.
$$

Therefore, any irreducible component of $C$ has $0\leq b\leq 1$. Moreover, tensoring Sequence \eqref{eMV} by $\sO_F(h_j)$ gives $h^0(\omega_{C_1 \cap C_2})\leq h^1(\omega_{C_2}(h_j))$ when $b=1$ and $h^0(\omega_{C_1 \cap C_2})\leq h^1(\omega_{C_2}(h_j))+1$ when $b=0$.

We are going to prove that $C$ consists of just one irreducible component with $b=1$. First of all, if there is no such component, $C$ would be the union of some irreducible components $C_i$, $i=1,\ldots,r$, representing $h_j^2$ in $A^2(F)$. Since they are disjoint pairwise, we have $r=1$ and therefore $Y$ would also represent $ah^2_j$ for some $a\geq 1$. However, this is impossible, since in this case $\sO_Y(th_j)\cong\sO_Y$ for all $t\in\mathbb{Z}$ and, in particular, we would get that $h^0(\sO_Y(-h_j))\neq 0$.

On the other hand, if $C$ contains two irreducible components $C_1$ and $C_2$ with classes $h_i^2+a_th_j^2$, $t=1,2$, again by means of the short exact sequence \eqref{eMV}, we see that $h^0(\omega_{C_1\cap C_2})=0$, that is $C_1$ and $C_2$ are disjoint. Since they are components of the connected curve $C$, there should exist a third irreducible component $Z$ of class $ah_j^2$ connecting them, and in particular intersecting their union in at least two points. But again, the exact sequence 
$$
0 \to \omega_{C_1\cup C_2} \oplus \omega_{Z} \to \omega_{C_1\cup C_2\cup Z } \to \omega_{(C_1\cup C_2) \cap Z} \to 0
$$
implies that $h^0(\omega_{(C_1\cup C_2) \cap Z})\leq 1$, a contradiction.

To complete the argument and exclude the case where $C$ is the reducible union of two curves, we can apply directly Lemma \ref{lCanonicalMultiple} to conclude that $C$ is an irreducible curve with class $h^2_i+ah^2_j$, $a\geq 0$.

Let us show now that if $a>0$ in the reduced structure, then $Y$ is quasi-primitive, namely $Y$ does not contain the first infinitesimal neighbourhood $C^{(1)}$ of $C$. Let us take the exact sequence
$$
0\to \sN^{\vee}_{C\mid F}\to\sO_{C^{(1)}}\to\sO_C\arr 0.
$$
By Lemma \ref{com-int}, we have $\sN^{\vee}_{C\mid F}\cong\sO_{\pp^1}(-1)\oplus\sO_{\pp^1}(1-2a)$. In particular, $h^1(\sO_{C^{(1)}}(-h_j))\neq 0$. Therefore, if $C^{(1)}\subset Y$, we would have a surjection
$$
H^1(\sO_Y(-h_j))\arr H^1(\sO_{C^{(1)}}(-h_j))\neq \{0\}
$$
contradicting that $h^1(\sO_Y(-h_j))=h^0(\omega_Y(h_j))=0$. 

In order to see that $Y$ is actually primitive, consider the filtration
\begin{equation}\label{filt}
    C=Y_0 \subset Y_1 \subset \dots \subset Y_k=Y
\end{equation}
and observe that from the short exact sequence
$$
0\to\frac{\sI_{Y_1}}{\sI_{C^2}}\to \mathcal{N}^{\vee}_{C\mid F}\cong\sO_{\pp^1}(-1)\oplus \sO_{\pp^1}(1-2a)\to \sO_{\pp^1}(\alpha)\to 0
$$
we obtain $\alpha\geq -1$. To exclude the equality, notice that it  would imply the short exact sequence
$$
0\to \sI_{C|Y_1}\cong\sO_C(-h_j)\to\sO_{Y_1}\to\sO_C\to 0
$$
and therefore $h^1(\sO_{Y_1}(-h_j))\neq 0$, contradicting again the hypotheses. Therefore, $\alpha\geq 0$. Computing
$$
0=\chi(\sO_Y(-h_j))=-\mathrm{mult}_C(Y)+1-p_a(Y).
$$
and putting this information in  Formula \eqref{eGenusMultiple}, we get $\alpha=d_t=0$ for all $t$, hence $Y$ is a primitive extension of type $\sO_C$.

Suppose now that $a=0$. In this case $Y_{red}=C$ is a line of type $h_i^2$. Therefore, as described in  Subsection \ref{sNonPrimitiveMultiple}, the given cohomological vanishings imply that $Y$ is a multiple structure of type $\sO_C$ of the form $Y = \pi_1^{-1}(Z)\cong Z\times \pp^1$, where $Z\subset\pp^2$ is  a 0-dimensional scheme supported on the point $p = \pi_1(C)$.
By Lemma \ref{lThickComInt}, $Y$ is a global complete intersection. If, for instance, the ideal $\sI_{C\mid F}$ of the line $C\subset F$ is globally generated by the variables $x_0,x_1$ from the first $\pp^2$, then $\sI_{Y\mid F}$ will be generated by two homogeneous polynomials $p(x_0,x_1)$ and $q(x_0,x_1)$ without common factors. 

To conclude, let us observe that if the degree of $p(x_0,x_1)$ or of $q(x_0,x_1)$ is equal to one, then $Y$ is a primitive extension. Otherwise, if both degrees are greater or equal than two, then  $C^{(1)}\subset C^{(\mathrm{min}\{n,m\}-1)}\subset Y$. In the latter case, multiple curves containing the first infinitesimal neighborhood of its reduced support are \emph{thick extensions} (see Definition \ref{dTrivialThick} and \cite[Section 4]{BF}).

$(2)\Rightarrow(1)$: Let $Y$ be a primitive extension or a complete intersection multiple structure of type $\sO_C$ as in $(2)$. Thanks to \eqref{eNormalMultiple} and \eqref{eNormalThick}, adjunction formula yields $\omega_Y\cong\sO_Y(-2h_j)$. 
As a last step, using recursively the short exact sequences
$$
0 \longrightarrow \sO_C^{\oplus r_k} \longrightarrow \sO_{Y_k} \longrightarrow \sO_{Y_{k-1}} \longrightarrow 0
$$
we conclude by induction that $h^0(\sO_Y(-h_j))=0$.
\end{proof}

\begin{remark}
Condition $(1)$ from the previous theorem clearly implies, by Serre's duality, that $h^0(\sO_Y(-h_j))=h^1(\sO_Y(-h_j))=0$. These weaker numerical conditions are not equivalent to the conditions from Theorem \ref{theor-equiv} as pointed out in Remark \ref{rTorsionFree}. Indeed, curves satisfying these two cohomological vanishings would be related, by means of a generalized version of Serre's correspondence, with instanton torsion-free sheaves, as defined in \cite{AnCa}, lying on the closure of the moduli space of $\mu$-stable $h_i$-'t Hooft bundles inside the moduli space of $\mu$-stable instanton sheaves. We believe that this approach could be very fruitful in general to understand the geometry of the moduli space of instantons. However, it requires the development of the theory of multiple curves over non integral curves and, therefore, it will be the aim of future investigation.
\end{remark}

The curves appearing in Theorem \ref{theor-equiv} can be also characterized via the projection $\pi_i$ using monads for instanton bundles. Let us start by recalling that by \cite[Theorem 1.1]{MMP} any instanton of charge $k$ is the cohomology of a monad of the form
\[
0 \arr \begin{matrix} H_1 \otimes \sO_F(-h_1)\\ \oplus \\  H_2 \otimes \sO_F(-h_2) \end{matrix} \xrightarrow{ \ \alpha \  }\begin{matrix} H_1^\vee \otimes \sG_1(-h_1)\\ \oplus \\ H_2^\vee \otimes \sG_2(-h_2)\end{matrix} 
\xrightarrow{ \ \beta \ } K \otimes \sO_F \arr 0,
\]
where $\sG_i$ is the pullback of the twisted cotangent bundle $\Omega_{\pp^2}(2)$ along the  natural projection $\pi_i: F\subset \pp(V_1) \times \pp(V_2) \to \pp(V_i)$ with $V_i \cong V_j^\vee$ and $H_1$, $H_2$ and $K$ are vector spaces of dimension $k$, $k$ and $2k-2$ respectively.
The display of the monad is given by 
\begin{equation*}
0 \to \sK\to  H_1^\vee \otimes \sG_1(-h_1)\oplus H_2^\vee \otimes \sG_2(-h_2)
\xrightarrow{\beta} K \otimes \sO_F \to 0
\end{equation*}
and
\begin{equation}\label{dis2}
0 \to  H_1 \otimes \sO_F(-h_1)\oplus H_2 \otimes \sO_F(-h_2)\to \sK  \xrightarrow{\pi} \sE \to 0.
\end{equation}
Now let us describe the maps $\alpha$ and $\beta$ appearing in the monad. The map $\alpha$ corresponds to an operator $A$ which can be described as the four-block matrix 
\[
A=
\begin{pmatrix}
  A_{11} & \rvline & A_{12} \\
\hline
  A_{21} & \rvline & A_{22}
\end{pmatrix}
\]
with $A_{ij} \in \Hom \Bigl(H_i \otimes \sO_F(h_i), H_j^\vee \otimes \sG_j(-h_j)\Bigr)$ and $1\le i,j\le 2$. Let us explicitly write the operators $A_{ii}$. Consider the Koszul complex
\[
0 \to \sO_{\pp^2}(-1) \to \Lambda^2 V_i^\vee  \otimes \sO_{\pp^2} \to \Lambda^1 V_i^\vee \otimes \sO_{\pp^2}(1) \to \sO_{\pp^2}(2) \to 0.
\]
In particular $\Hom(\sO_{\pp^2}, \Lambda^2 V_i^\vee \otimes \sO_{\pp^2}) \cong \Hom(\sO_{\pp^2}, \Omega_{\pp(V_i)}(2))$, thus $A_{ii}$ is determined by an operator
\[
A_{ii}:H_i \arr H_i^\vee \otimes \Lambda ^2 V_i^\vee.
\]
Now consider $A_{ij}$, corresponding to an element in $\Hom(H_i \otimes \sO_F, H_j^\vee \otimes \sG_j(h_i-h_j))$. Notice that 
$$H^0(\sG_j(h_i-h_j))\cong H^0(\sO_{\pp(V_i)})\otimes H^0(\sO_{\pp(V_j)})\cong \mathbb{C},$$ thus $A_{ij}$ corresponds to an operator
\[
A_{ij} : H_i \arr H_j^{\vee}.
\]
Now we deal with the map $\beta$. It corresponds to the column of operators $\begin{pmatrix} B_1 \\   B_2 \end{pmatrix}$. Using the Koszul complex, with a similar argument as before, we obtain that they can be described as
\begin{equation*}
B_i: H_i^\vee \otimes \Lambda^2 V_i^\vee \arr K \otimes V_i^\vee
\end{equation*}
Let us suppose that $\sE$ is an instanton bundle such that $H^0(\sE(h_i))\neq 0$. Now we want to describe the zero locus of a section. We have the exact triples
\[
0 \to \sK(h_i)\to \sG_1(-h_1+h_i)^{\oplus k}\oplus \sG_2(-h_2+h_i)^{\oplus k}
\xrightarrow{\beta} \sO_F(h_i)^{\oplus 2k-2} \to 0
\]
and
\[
0 \to \sO_F(-h_1+h_i)^{\oplus k}\oplus \sO_F(-h_2+h_i)^{\oplus k} \to \sK(h_i)  \to \sE(h_i) \to 0
\]
Consider a section $t_i\in H^0(\sK(h_i))$. It can be identified with an element $a=(a',a'') \in K_j^\vee \oplus K_i^\vee\otimes \Lambda^2(V_i) $ satisfying $a \wedge B=0$. Suppose $i=1$, the other case being completely analogous. The evaluation at a point $(\langle v_1\rangle,\langle v_2\rangle)\in F$ is given by
\[
\bigr(h_1 \otimes (v_i^1)^\ast \wedge (v_j^1)^\ast, h_2\bigl) \arr k_1 \otimes \Bigl(v_1((v_i^1)^\ast)(v_j^1)^\ast-v_1((v_j^1)^\ast)(v_i^1)^\ast\Bigr) + k_2 \otimes v_2
\]
via the canonical map $V_i^\vee\otimes \Lambda^2 V_i \arr V_i$. In particular we see that $a \wedge B=0$ precisely at the points $(\langle v_1\rangle,\langle v_2\rangle)\in F$ such that 
$$\alpha v_2+\beta\Bigl(v_1((v_i^1)^\ast)(v_j^1)^\ast-v_1((v_j^1)^\ast)(v_i^1)^\ast\Bigr)=0,$$ which fill a line in $\pp(V_j)$.

Using the previous description, we are able to present a different proof of the shape of the irreducible components of the $0$-locus of a section of a $h_i$-'t Hooft bundle.
\begin{proposition}\label{0locus}
Let $\sE$ be a $\mu$-stable, $h_i$-'t Hooft instanton bundle of charge $k$ and let $s_i \in H^0(\sE(h_i))$ be a non-zero section. Then $(s_i)_0=Y$ is a purely two codimensional subscheme of $F$ whose reduced structure $Y_{red}$ is a disjoint union of smooth rational curves $Y_a$ in classes $h_i^2+ah_j^2$ with $a \ge 0$.
\end{proposition}
\begin{proof}
Since $\sE$ is stable, any component of $V(s_i)=(s_i)_0$ has pure dimension one. Let $X$ be a connected union of components. By \eqref{dis2}, the map $\pi$ is surjective at the level of global sections. In particular there exists $t_i \in H^0(\sK(h_i))$ such that $\pi(t_i)=s_i$. Moreover we have $\pi({t_i}_{|_{X_{red}}})=0$. Since $X_{red}$ is connected we have $h^0(\sO_{X_{red}})=1$ and we can find $u_i\in H_i$ such that $t_i'=t_i-A_{ii}(u_i)$ vanishes on $X_{red}$, thus $X_{red}\subseteq V(t_i')$. We showed that $V(t_i')$ is either a point or a line on $\pp(V_j)$, and their pullbacks on $F$ correspond to the desired curves. 
\end{proof}

We will now deal with the description of the Hilbert scheme of curves appearing as the zero loci of a $h_i$-'t Hooft bundle.

{\bf{Notation:}} Let us introduce the following notation. Let $\bar{a}=(a_1,\dots,a_m)$ be a multi-index such that $a_t \in \mathbb{Z}$ and $0 \le a_1 \le \dots \le a_m$. The curve 
\begin{equation}\label{0schemeflag}
Y^{i}_{\bar{a}}=Y^{i}_{(a_1,\dots,a_m)}
\end{equation}
is the disjoint union of possibly non-reduced locally complete intersection curves, each of which is supported on a smooth rational complete intersection curve $C^{i}_{a_t}$ in the class $h_i^2+a_th_j^2$ ($1\le t \le m$) with $a_t \ge 0$. If non-reduced, it is a multiple structure as described in Theorem \ref{theor-equiv}. The curves $Y^{i}_{\bar{a}}$ will be represented as lists of components $C^{i}_{a_t}$ in which every component supporting a multiple structure of multiplicity $r$ appears $r$ times. In particular we denote by $Y^{i}_{a_t}$ the multiple structure supported on the reduced curve $C^{i}_{a_t}$.

From Theorem \ref{theor-equiv} we deduce that given a section $s_i \in H^0(\sE(h_i))$ for a charge $k$ instanton,  the zero locus $Y=(s_i)_0$ is $Y=Y^{i}_{\bar{a}}$ for some multi-index $\bar{a}$. If it is clear from the context, we will sometimes drop the index $(i)$. Given a multi-index $\bar{a}$, we denote by 
\[
\ell(\bar{a}):=\#\{\text{indices of $\bar{a}$ which are equal to $0$}\}.
\]
Geometrically, this is the number of lines (counted with multiplicities) appearing in $Y$.

\begin{lemma}\label{lIndices}
Let $\sE$ be a charge $k$ instanton bundle and let $Y_{\bar{a}}^i$ be the vanishing locus of $s_i\in H^0(\sE(h_i))$. Then 
\begin{equation}\label{eIndices}
\bar{a}=(a_1,\dots,a_{k+1}), \qquad \sum_{i=1}^{k+1}{a_{i}}=k, \qquad 1 \le \ell(\bar{a})\le k.
\end{equation}
\end{lemma}
\begin{proof}
The zero locus $Y^i$ of a section of $\sE(h_i)$ is a curve of the form $Y^i=Y_{\bar{a}}^i$ thanks to Theorem \ref{theor-equiv} and Proposition \ref{0locus}. The statements then follow directly from the fact that $Y_{\bar{a}}^i$ represents the class $$c_2(\sE(h_i))=kh_1h_2+h_i^2=(k+1)h_i^2+kh_j^2$$ in $A^2(F)$.
\end{proof}

Let us now compute the normal bundle of such curves. By Proposition \ref{com-int}, Equation \eqref{eNormalMultiple} and Equation \eqref{eNormalThick} we have
\begin{equation}\label{eNormal2}
\sN_{C^{i}_{a_t}|F} \cong
\begin{cases}
\sO_{\pp^1}(1) \oplus \sO_{\pp^2}(2a_t-1) \ & \text{if} \ a_t \ge 1, \\
{\sO_{\pp^1}}^{\oplus 2} \ & \text{if} \ a_t=0,
\end{cases}
\end{equation}
in the reduced case, and
\begin{equation}\label{eNormal3}
\sN_{Y^{i}_{a_t}|F} \cong\sO_{Y^{i}_{a_t}}\oplus \sO_{Y^{i}_{a_t}}(2h_i) \  \text{with} \ a_t \ge 0
\end{equation}
if $Y^{i}_{a_t}$ is a multiple structure as described in Theorem \ref{theor-equiv}. In the following proposition we study the Hilbert scheme of the curves described above.

\begin{proposition}\label{pHilbertComponents}
Let $H \subset \mathcal{H}:=\mathrm{Hilb}^{(2k+1)t+(k+1)}(F)$ be the Hilbert scheme of curves $Y$ of degree $2k+1$ and arithmetic genus $p_a(Y)=-k$ satisfying $\omega_Y\cong \sO_Y(-2h_j)$ and $h^0(\sO_Y(-h_j))=0$. Then $H$ consists of at least $k$ irreducible components.
\end{proposition}

\begin{proof}
By Theorem \ref{theor-equiv} we have $Y=Y^{i}_{\bar{a}}$ for some $\bar{a}$. By a direct computation of Chern classes, one obtains $h^0(\sO_Y(-h_i))=\ell(\bar{a})$. 

Let $H_{\ell}\subset H$ be the locus of curves $Y_{\bar{a}}$ satisfying $\ell(\bar{a})=\ell$. Consider two multi-indices $\bar{a}$ and $\bar{a}'$ and set $\ell(\bar{a})=\ell$ and $\ell(\bar{a}')=\ell'$. Suppose $\ell(\bar{a})<\ell(\bar{a}')$.  If $H_{\ell}$ and $H_{\ell'}$ lie in the same irreducible component, since the cohomology function is lower semi-continuous, $\dim(H_{\ell'}) < \dim(H_{\ell})$. Thanks to \eqref{eNormal2}, \eqref{eNormal3} and \eqref{strSheaf2}, we have 
\begin{equation}\label{eCohomologyNormal}
h^1(\sN_{Y_{\bar{a}}|F})=0 \qquad \text{and} \qquad h^0(\sN_{Y_{\bar{a}}|F})=4k+2
\end{equation}
for all $\bar{a}$, thus $\dim(H_{\ell'})=\dim(H_{\ell})$ which leads to a contradiction. In particular if $\ell\neq\ell'$ then $H_{\ell}$ and $H_{\ell'}$ live in two different irreducible components. The proof is complete by noticing that $1\le\ell(\bar{a})\le k$.
\end{proof}

Now we deal with the inverse problem, i.e. we show that starting from a scheme as in Theorem \ref{theor-equiv} we actually obtain an instanton bundle. In order to do so, we will use Serre's correspondence between curves and rank two vector bundles on $F$. 

Notice that for each connected component we have $\det \sN_{Y^{i}_{\bar{a}}|F}\cong \sO_F(2h_i)\otimes\sO_{Y^{i}_{\bar{a}}}$ thanks to \eqref{eNormal2} and \eqref{eNormal3}, thus 
$$
\det \sN_{Y|F}\cong \sO_F(2h_i)\otimes\sO_{Y},
$$
i.e, the determinant of the normal bundle of $Y$ is extendable on $F$. Since $h^2(\sO_F(-2h_i))=0$, there exists a vector bundle $\sF$ on $F$ with a section $s$ vanishing along $Y$ with $c_1(\sF)=2h_i$ and $c_2(\sF)=Y$. Thus $\sE=\sF(-h_i)$ has $c_1(\sE)=0$, $c_2(\sE)=c_2(\sF)-h_i^2$ and thanks to \cite[Theorem 1]{Ar} it fits into the exact triple
\begin{equation}\label{extserreflag}
0 \rightarrow \sO_F(-h_i)\to \sE \to \sI_{Y}(h_i)\to0.
\end{equation}
In the following proposition, we show that the vector bundles arising in this way are actually instanton bundles.

\begin{proposition}\label{serreexistenceflag}
Let $\sE$ be a vector bundle with $c_1(\sE)=0$, $c_2(\sE)=kh_1h_2$ and $k\geq 2$. Suppose $\sE(h_i)$ has a section whose scheme of zeros is $Y=Y_{\bar{a}}$ as in construction  \eqref{0schemeflag} which satisfies \eqref{eIndices}, i.e. $\sE$ fits into the exact triple \eqref{extserreflag}.

Then $\sE$ is a $\mu$-stable instanton bundle of charge $k$. Moreover, we have 
\[
\dim \Ext_F^1(\sE,\sE)=8k-3 \qquad \text{and} \qquad\Ext_F^2(\sE,\sE)=\Ext_F^3(\sE,\sE)=0.
\]
\end{proposition}

\begin{proof}
First of all notice that by construction we have $c_1(\sE)=0$ and $c_2(\sE)=kh_1h_2$.

Taking the cohomology of \eqref{extserreflag}, we obtain $h^0(\sE)=h^0(\sI_{Y|F}(h_i))=0$ because $Y$ contains at least two disjoint components and any two curves $Y_{a_t}$ in a surface $S\in |\sO_F(h_i)|$ always intersect. Tensoring \eqref{extserreflag} by $\sO_F(-h)$ we have $h^1(\sE(-h))=h^1(\sI_{Y|F}(-h_j))$ with $j \neq i$. Considering the defining sequence of the ideal $\sI_{Y|F}$ tensored by $\sO_F(-h_j)$ we obtain $h^1(\sI_{Y|F}(-h_j))=h^0(\sO_Y(-h_j))=0$ because of Theorem \ref{theor-equiv}.

Now we prove the $\mu$-stability of $\sE$. Thanks to the Hoppe's criterion \cite[Theorem 3]{JMPS} $\sE$ is $\mu$-stable if and only if $h^0(\sE(-D))=0$ for each divisor $D$ such that $Dh^2\geq 0$. Let us take such a divisor $D=d_1h_1+d_2h_2$ with $d_1+d_2 \geq 0$ and consider the short exact sequence
\[
0 \to \sO_F(-D-h_i)\to \sE(-D) \to \sI_{Y|F}(-D+h_i)\to 0.
\]
Now $h^0(\sI_{Y|F}(-D+h_i))\leq h^0(\sO_F(-D+h_i))$, so it is clear that $h^0(\sI_{Y|F}(-D+h_i))=0$ whenever $d_j>0$ or $d_i>1$. In these cases we have $h^0(\sE(-D))=0$. It remains to study the cases $D=h_i$ and $D=-h_j+h_i$. In both cases we obtain $h^0(\sE(-D))=h^0(\sI_{Y|F}(-D+h_i))=0$ because $Y$ contains at least two reduced disjoint components, thus $\sE$ is $\mu$-stable.

Finally, we prove the part of the statement regarding the $\Ext$ groups of $\sE$. Since $\sE$ is $\mu$-stable, it is simple; thus, we have $\Hom_F(\sE,\sE)\simeq \cc$ and $\Ext_F^3(\sE,\sE)=0$. It now suffices to show that $\Ext^2(\sE,\sE)=0$ to compute directly the dimension of $\Ext_F^1(\sE,\sE)$ by Riemann-Roch. Consider the short exact sequence \eqref{extserreflag} and tensor it by $\sE\cong \sE^\vee$. Taking cohomology we have
\[
H^2(\sE(-h_i))\to \Ext_F^2(\sE,\sE) \to H^2(\sE \otimes \sI_{Y|F}(h_i)).
\]
From \eqref{extserreflag} we obtain $H^2(\sE(-h_i))\cong H^2(\sI_{Y|F}) \cong H^1(\sO_Y)\cong 0$ thanks to sequences \eqref{strSheaf2} and \eqref{eStrThick}, because $Y$ is the disjoint union of possibly multiple structures on smooth rational curves $C_{a_t}$. So $\Ext_F^2(\sE,\sE)=0$ as soon as $H^2(\sE \otimes \sI_{Y|F}(h_i))$ vanishes.  In order to show this vanishing, let us take the short exact sequence \eqref{extserreflag} and tensor it by $\sO_F(h_i)$. Taking cohomology we obtain $h^2(\sE(h_i))=h^2(\sI_{Y|F}(2h_i))$. Now if we tensor
\[
0\to\sI_{Y|F}\to \sO_F \to \sO_Y \to 0
\]
by $\sO_F(2h_i)$ we have $h^2(\sI_{Y|F}(2h_i))=h^1(\sO_F(2h_i)\otimes \sO_Y)=0$ since $\sO_F(2h_i)$ restricts to each component of $Y$ to a non-negative degree line bundle. Thus we have $h^2(\sE(h_i))=0$. Now if we take the cohomology of the defining sequence of $\sI_{Y|F}$ tensored by $\sE(h_i)$ we have 
\[
h^2(\sE(h_i)\otimes \sI_{Y|F})\leq h^1(\sE(h_i)\otimes\sO_Y).
\]
But now using the fact that $\sE\otimes \sO_Y(h_i)\cong \sN_{Y|F}$ we have $h^1(\sE(h_i)\otimes\sO_Y)=0$ thanks to \eqref{eNormal2} and \eqref{eNormal3}. Hence $h^2(\sE(h_i)\otimes \sI_{Y|F})=0$ and finally we obtain $\Ext_F^2(\sE,\sE)=0$. To compute the dimension of $\Ext_F^1(\sE,\sE)$ we use Riemann-Roch. Since $c_1(\sE\otimes \sE^\vee)=c_3(\sE\otimes \sE^\vee)=0$ and $c_2(\sE \otimes \sE^\vee)=4c_2(\sE)$ we have
\[
\dim \Ext_F^1(\sE,\sE)=h^0(\sE \otimes \sE^\vee)+h^2(\sE \otimes \sE^\vee)-\chi(\sE \otimes \sE^\vee)=4c_2(\sE)(h_1+h_2)-3,
\]
thus $\dim \Ext_F^1(\sE,\sE)=8k-3$.
\end{proof}

We conclude this section finding a bound on the maximal dimension of the cohomology group $H^0(\sE(h_i))$ of an instanton bundle $\sE$.

\begin{lemma}\label{lemmacurve2}
Let $\sE$ be a $h_i$-'t Hooft bundle of charge $k$ and let $Y_{\bar{a}}^{i}$ be the  reduced curve associated to the vanishing locus of $s_i\in H^0(\sE(h_i))$. Then 
$$
h^0(\sI_{Y_{\bar{a}}^{i}|F}(2h_i))=
\begin{cases}
    2 \ \text{if $Y_{\bar{a}}^{i}\cong Y_{(0,1)}^{i}\cong C^{i}_0 \cup C^{i}_1$}\\
     0 \ \text{otherwise}.
     \end{cases}
     $$
\end{lemma}
\begin{proof}
The image of the projection $\pi_i:F\rightarrow \pp^2$ restricted to $Y^i_{\bar{a}}$ is the union $\bigcup_{t=l(\bar{a})+1}^{m}C_{a_t} \subset \pp^2$ of $m-\ell(\bar{a})$ rational curves of degree $a_t$ plus $\ell(\bar{a})$ distinct points not belonging to the curves $C_{a_t}$. Thanks to restrictions imposed by Lemma \ref{lIndices}, we see that as soon as $\ell(\bar{a})>1$ there should exist $t$ such that $a_t\ge 2$. From this observation, the lemma follows in a straightforward manner.
\end{proof}

\begin{proposition}\label{pBoundSections}
Let $\sE$ be an instanton bundle of charge $k$ on F. 
\begin{itemize}
    \item If $k=1$, then $h^0(\sE(h_i))=3$.
    \item If $\sE$ is $\mu$-stable and $k\geq 2$ then $h^0(\sE(h_i))\leq 1$ for all $i$;
    \item If $\sE$ is properly $\mu$-semistable and $k \ge 2$ then $h^0(\sE(h_i))=0$ for all $i$.
\end{itemize}
\end{proposition}
\begin{proof}

Let $\sE$ be a $\mu$-stable instanton bundle of charge $k$ such that $h^0(\sE(h_i))\neq 0$ and let $Y$ be the vanishing locus of a section $s\in H^0(\sE(h_i))$. By the inequality
$$h^0(\sI_Y(2h_i))\leq h^0(\sI_{Y_{red}}(2h_i)),$$
Lemma \ref{lemmacurve2} and the exact sequence \eqref{eSerreHooft},  it is enough to deal with the case $Y_{red}=Y^i_{(0,1)}$. In this case, $Y=L\cup D$, where $D$ is a multiple structure of multiplicity $k$ with support a smooth conic. In this case, $h^0(\sI_Y(2h_i))\neq0$ if and only if $Y$ and contained in a Hirzebruch surface representing $h_i$. Following the same approach of the proof of Proposition \ref{com-int}, $D$ is a complete intersection of type $h_i,kh_j$. Then $D$ satisfies $\omega_D\cong\sO_D(-2h_j)$ (restriction imposed by Theorem \ref{theor-equiv} and the fact that $Y$ is the vanishing locus of a section of $\sE(h_i)$) only when $k=1$. In this case we get $h^0(\sE(h_i))=3$.

On the other hand, if $\sE$ is properly $\mu$-semistable, then the statement follows directly from \cite[Proposition 3.5]{MMP}.
\end{proof}

\section{$h$-'t Hooft bundles}\label{sHtHooft}
In this section we deal with the existence of $h$-'t Hooft instantons  on the flag variety $F$, i.e. instanton bundles $\sE$ such that $h^0(\sE(h))\neq 0$. Let us start with some preliminary observations.
\begin{enumerate}
    \item Any $h_i$-'t Hooft instanton bundle $\sE$ is also $h$-'t Hooft, since $H^0(\sE(h_i))$ injects in $H^0(\sE(h))$. In this section we will be interested in \emph{proper} $h$-'t Hooft instantons, namely those for which $h^0(\sE(h_i))=0$ for $i=1,2$.    
    \item From Euler formula \eqref{euler-general-twist} any instanton bundle $\sE$ satisfies $\chi(\sE(h_i))=6-3k$, so any instanton bundle of charge $k=1$ is $h_i$-'t Hooft for both $i$ and in particular  $h$-'t Hooft. For charge two instantons $\chi(\sE(h_i))=0$; therefore, if $\sE$ is a proper $h$-'t Hooft instanton of charge two, then $\sE(3h_i+h_j)$ is an Ulrich bundle with respect to $\sO_F(2h_i+h_j)$.
    \item From Euler formula  \eqref{euler-general-twist} we have $\chi(\sE(h))=16-4k>0$ for $k=1,2,3$ so instanton bundles of these charges are always $h$-'t Hooft.
\end{enumerate}
Thanks to the item $(3)$ above and \cite[Theorem 1.1]{MMP} the existence of $h$-'t Hooft instanton of charge $k$ is guaranteed if $k\le 3$. In the next Theorem we deal with the existence of the curves on $F$ that will correspond, by means of Serre's correspondence, to $h$-'t Hooft bundles for any positive charge.

\begin{theorem}
\label{tElliptic}
For any $k\geq 1$, there exists a family $\mathfrak{H}_k$ of dimension $4k+12$ of integral smooth elliptic curves $Y$ with class $(k+3)h_1h_2$ which are not contained in a hyperplane section of $F$ (i.e. non-degenerate). Each element $Y$ of this family corresponds to a smooth point of the Hilbert scheme $\mathrm{Hilb}^{(2k+6)t}(F)$ satisfying $h^1(\sT_F\otimes \sO_Y)=0$ where $\sT_F$ denotes the tangent bundle of $F$. Moreover, for $k\geq 2$, there exists $Y\in \mathfrak{H}_k$ such that $h^0(\sI_{Y\mid F}(h_i+2h_j))=0$ for $i \neq j$.
\end{theorem}

\begin{remark}
Let us say a few words on the hypothesis of Theorem \ref{tElliptic}, in particular why it requires the condition $k\ge 2$ to obtain the vanishings $h^0(\sI_{Y\mid F}(2h_1+h_2))=h^0(\sI_{Y\mid F}(h_1+2h_2))=0$.

\begin{enumerate}
\item The case $k=0$ corresponds to hyperplane sections in the del Pezzo surface $S=F\cap H$ of degree 6. But they are degenerate. They correspond to the vector bundle $\sO_F\oplus\sO_F$. In this case, the associated elliptic curve $Y$ satisfies $H^0(\sI_{Y\mid F}(h))\neq 0$ and therefore for any $q\in F$ we can find a surface in the class $2h_i+h_j$ containing $Y$ and $q$. So we can not start an induction process from $k=0$. This agrees with the fact that for $k=1$, any $h$-'t Hooft instanton is also  $h_i$-'t Hooft.

\item In the same spirit, in the case $k=1$, if one starts the argument with an elliptic curve $Y$ with representative $4h_1h_2$ in the Chow ring (which corresponds to an Ulrich bundle with respect to $\sO_F(h)$),  then $h^0(\sI_{Y\mid F}(h_i+2h_j)=3$, so a priori it could happen also that for any $q\in F\backslash Y$ there exists a section of 
$\sI_{Y\mid F}(h_i+2h_j)$ vanishing on $q$. This is also an obstruction to the induction argument used in the proof of Theorem \ref{tElliptic}.

\item The vanishing $h^1(\sT_F\otimes \sO_{Y})=0$ is a necessary technical condition required to perform the inductive construction we explain hereafter.
\end{enumerate}
\end{remark}

\emph{Proof of Theorem \ref{tElliptic}.}
The base case $k=1$ was already proved in \cite[Theorem 6.6]{MMP}.  We will start dealing with the case $k=2$ with the added requirement that $h^0(\sI_{Y\mid F}(h_i+2h_j))=0$ for $i \neq j$ and then proceed by induction on $k$. 

So let be $S\subset F$ a smooth surface of class $h_1+2h_2$. Thanks to Lemma \ref{lRestrictionPic}, the restriction of the second projection $\pi_{2\mid S}:S\rightarrow\pp^2$ bestows $S$ with the structure of the blowup of $\pp^2$ at $7$ sufficiently general points. Let $C$ be a generic element of the linear system on $S$ of type $5l-3e_1-2e_2-2e_3-e_4-e_5-e_6$. It is a smooth elliptic curve of degree 10. Its class in the Chow ring of $F$ is $5h_1h_2$.
A standard argument shows that this curve $C$ can be deformed inside the Hilbert scheme $\mathrm{Hilb}^{10t}(F)$ to an elliptic curve $Y$ such that $h^0(\sI_{Y\mid F}(h_i+2h_j))=0$ for $i \neq j$. Indeed, from the short exact sequence of normal bundles
$$
0\rightarrow\sN_{C|S}\rightarrow\sN_{C|F}\rightarrow\sN_{S|F}\otimes \sO_C\rightarrow 0
$$
one can see that $h^1(\sN_{C|F})=0$, thus $h^0(\sN_{C|F})=20$ by Riemann-Roch. On the other hand, the family of surfaces of type $h_i+2h_j$ has dimension $14$ and their linear systems $|5l-3e_1-2e_2-2e_3-e_4-e_5-e_6|$ (corresponding to quintic curves in $\pp^2$ passing through $6$ given points with multiplicities $3,2,2,1,1,1$) have dimension $5$. Therefore, a general deformation $Y$ of $C$ inside $\mathrm{Hilb}^{10t}(F)$ will satisfies our requirements.

The previous argument completes the case $k=2$. Now we will use an induction argument to treat the remaining case $k\geq 3$. So let us suppose that the family $\mathfrak{H}_k$ of curves satisfying the conclusion of the theorem has been constructed for a certain $k\ge 2$ and let $Y\subset F$ be a smooth elliptic curve represented by a general point of the family $\mathfrak{H}_k$. Let $q\in F\backslash Y$ be a general point and let $Y'=Y \cup  C$ be the reduced and reducible curve  with $C$ a smooth conic passing through $q$ such that $Y\cap C:=\{p\}$ is a single point. Notice that $Y'$ has a unique nodal singularity at $p$. Moreover the Chern class of $Y'$ is $(k+4)h_1h_2$, and the arithmetic genus of $Y'$ is
$$
p_a(Y')=p_a(Y)+p_a(C)-1+\textrm{card}(Y\cap C)=1.
$$
Since $h^0(\sI_{Y\mid F}(h_i+2h_j))=0$ for $i\neq j$, a fortiori the same holds for $Y'$ and therefore by semicontinuity the same will be true for a general deformation of $Y'$. 

Let us consider the Hilbert scheme $\mathrm{Hilb}^{(2k+8)t}(F)$ of curves in $F$ of degree $2k+8$ and arithmetic genus $1$. Let $[Y']\in \mathrm{Hilb}^{(2k+8)t}(F)$ be the class of one of these curves and let $\mathscr{H}\subset F\times \mathrm{Hilb}^{(2k+8)t}(F)$ the universal family. 
We now show that $Y'$ can be smoothly deformed inside $\mathrm{Hilb}^{(2k+8)t}(F)$. Recall that, if we consider a non-singular projective variety $X\subset \pp^N$ and  a curve $Z\subset X$, then the differentiation map $d:\sI_Z/\sI_Z^2\rightarrow\Omega^1_X\otimes\sO_Z$ gives rise to a natural map
$$
\phi:(\sT_{X})_{\mid Z}\rightarrow \sN_{Z\mid X}.
$$
Let $S$ be the set of singular points of $Z$. The cokernel of $\phi$ is supported on $S$ and it is the $T^1$-functor of Lichtenbaum-Schlessinger denoted by $T^1_Z$. In the case of a nodal curve, $T^1_Z$ is isomorphic to its restriction $T^1_S$ to $S$.

\begin{proposition}\label{hh1}
Let $Z\subset X$ be a nodal curve with Hilbert polynomial $p(t)$. Assume that $h^1(\sN_{Z\mid X})=0$ and that for each singular point $p\in S$, the natural map $H^0(\sN_{Z\mid X})\rightarrow H^0(T^1_p)$ is surjective. Then the Hilbert scheme $\mathrm{Hilb}^{p_Z(t)}(X)$ is smooth at $Z$ and $Z$ can be deformed to a smooth curve inside $\mathrm{Hilb}^{p_Z(t)}(X)$.
\end{proposition}

\begin{proof}
See \cite[Proposition 1.1]{HH} and \cite[Theorem 6.3]{Sernesi}.
\end{proof}

To apply the previous proposition, we are going to use the following result:

\begin{lemma}
Let $Z\subset X$ be a nodal curve such that $h^1((\sT_{X})_{\mid Z})=0$. Then $h^1(\sN_{Z\mid X})=0$ and $Z$ is smoothable.
\end{lemma}

\begin{proof}
The proof relies on \cite[Corollary 1.2]{HH}. Let $\sN':=Im((\sT_{X})_{\mid Z}\rightarrow \sN_{Z\mid X})$. Then $h^1(\sN')=0$ and from
$$
0\rightarrow \sN'\rightarrow \sN_{Z\mid X}\rightarrow \sT^1_S\rightarrow 0,
$$
we obtain that $h^1(\sN_{Z\mid X})=0$ and $H^0(\sN_{Z\mid X})\rightarrow H^0(\sT^1_S)$ is surjective. Then we conclude by Proposition \ref{hh1}.
\end{proof}

Now we can conclude the proof of Theorem \ref{tElliptic}. Let $Y$ be a smooth elliptic curve in $(k+3)h_1h_2$ and $C$ a smooth conic such that $Y'=Y\cup C$ has a single node at the intersection point $p$ of $Y$ and $C$. Then $Y'$ is a non-degenerate curve from the class $(k+4)h_1h_2$. On the other hand, we have
$$
0\rightarrow \sO_{Y'}\rightarrow\sO_Y\oplus\sO_C\rightarrow \sO_p\rightarrow 0.
$$
If we tensor the previous sequence by $\sT_F$, we obtain
$$
0\rightarrow \sT_F\otimes \sO_{Y'} \rightarrow\sT_F\otimes \sO_Y\oplus\sT_F\otimes\sO_C \stackrel{\phi}{\rightarrow} \sT_{F,p}\rightarrow 0.
$$
Now, since $h^1(\sT_F\otimes\sO_Y)=0$ by the induction hypothesis, the map $\phi$ is surjective at the level of global sections since $\sT_F$ is globally generated and $h^1(\sT_F\otimes\sO_C)=0$, as follows from the short exact sequence
$$
0\rightarrow \sT_C\cong\sO_C(2)\rightarrow\sT_{F}\otimes \sO_C\rightarrow \sN_{C\mid F}\cong\sO_C(1)^{\oplus 2}\rightarrow 0.
$$
Finally, we compute the dimension of this family. The Riemann-Roch formula applied to $\sN_{Y\mid F}$ yields $h^0(\sN_{Y\mid F})=2\deg(Y)=4k+12$ and the proof is complete.\qed

\medskip
Via Serre's correspondence (cf. \cite[Theorem 1]{Ar}), we get the following
\begin{corollary}\label{htHooft}    
There exist $\mu$-stable, $h$-'t Hooft instantons of charge $k$ for any $k\geq 1$. For $k\geq 2$, there exist proper ones.
\end{corollary}
\begin{proof}
The case $k=1$ has been proved in \cite[Theorem 6.6]{MMP}. For the case $k\ge 2$ consider the elliptic curves $Y$ constructed in Theorem \ref{tElliptic}. Serre's correspondence yields a vector bundle $\sE$ fitting into the short exact sequence
\begin{equation}\label{eSerreElliptic}
0 \to \sO_F(-h) \to \sE \to \sI_{Y|F}(h)\to 0.
\end{equation}
We show that $\sE$ is actually a $\mu$-stable instanton bundle. The instantonic condition is trivially satisfied since $h^1(\sO_F(-2h))=h^1(\sI_{Y|F})=0$. It remains to show that $\sE$ is  $\mu$-stable. Using Hoppe's criterion \cite[Theorem 3]{JMPS}, it is enough to check that, given a divisor $D=ah_1+bh_2$, then $h^0(\sE(D))=0$ for all $(a,b)$ with $Dh^2=3(a+b) \le 0$. After tensoring the exact sequence \eqref{eSerreElliptic} by $\sO_F(D)$, we observe that $h^0(\sO_F(-h+D))=0$ for all $D$ of non-positive degree. Moreover the natural injection 
$$
0 \to\sI_{Y|F}(D+h) \to \sO_F(D+h)
$$ 
yields $h^0(\sI_{Y|F}(D+h))=0$ for all the couples $(a,b)$ with $a+b \le 0$ which are different from $(0,0)$, $(-1,0)$, $(0,-1)$, $(1,-1)$ and $(-1,1)$. The vanishings of $h^0(\sI_{Y|F}(D+h))$ for these exceptional cases are a consequence of the fact that the curves constructed in Theorem \ref{tElliptic} satisfy $h^0(\sI_{Y|F}(h_i+2h_j))=0$ for $i\neq j$.
\end{proof}

\section{Special 't Hooft bundles}\label{sSpecialtHooft}
In this section we introduce the notion of special instanton bundles.
\begin{definition}
An instanton bundle $\sE$ is called a \textit{special 't Hooft bundle} if and only if  $h^0(\sE(h_1))$ and $h^0(\sE(h_2))$ are both different from zero.
\end{definition}

Recall that  in the case of instanton bundles on the projective space $\pp^3$, it holds that $h^0(\sE(1))\le 2$  for all $\sE$ and $h^0(\sE(1))=0$ for $\sE$ generic. Those reaching this bound are called special instanton bundles in \cite{BeFr,BT,HN}. We decided to use this terminology in our setting to highlight the analogy contained in the following remark.

\begin{remark}
\label{rdelPezzo}
An instanton bundle $\sE$ is a special 't Hooft bundle if and only if the subschemes $Y^1$, $Y^2$ associated to non-zero sections of $H^0(\sE(h_1))$, $H^0(\sE(h_2))$ are contained in a single divisor of type
$h_1+h_2$. Indeed from the short exact sequence
$$
0\rightarrow \sO_F\rightarrow\sE(h_1)\rightarrow\sI_{Y^1|F}(2h_1)\rightarrow 0
$$
associated to a section of $\sE(h_1)$, after twisting by $\sO_F(-h_1+h_2)$, we get
$$
0\rightarrow \sO_F(-h_1+h_2)\rightarrow\sE(h_2)\rightarrow\sI_{Y^1|F}(h)\rightarrow 0
$$
from where the claim follows. Let us notice that this situation is analogous to the one for classical instantons on $\pp^3$, where the condition for being special is equivalent to requiring that  the curves associated to sections of $\sE(1)$ are contained in a single smooth quadric surface.

In the case of special instanton bundles, it is natural to ask if there is any relation between the zero loci of the sections $s_i \in H^0(\sE(h_i))$ for $i=1,2$. 
\end{remark}

\begin{proposition}\label{dep}
The dependence locus of two sections $s_1 \in H^0(\sE(h_1))$, $s_2\in H^0(\sE(h_2))$ is a sextic surface $S_{(1,1)} \subset F$. Moreover, the intersection of the zero locus $Y_1$, $Y_2$ of these two sections is contained in the singular locus of $S_{(1,1)}$.

\end{proposition}
\begin{proof}
From Remark \ref{rdelPezzo}, it follows that the image of the non-zero global section $s_2\in H^0(\sE(h_2))$ in $H^0(\sI_{Y^1|F}(h))$ corresponds to a surface $S_{(1,1)}$. This can be seen by gathering together the exact sequences given by the sections $s_1,s_2$ in the following commutative diagram: 
\begin{equation}\label{eGeneralEvaluation}
\xymatrix{
& 0 \ar[d] & 0 \ar[d]\\
& \sO_F(h_2) \ar[r]^\cong \ar[d] & \sO_F(h_2) \ar[d]\\ 
0 \ar[r] & \sO_F(h_1)\oplus\sO_F(h_2)\ar[r]^-{(s_2,s_1)} \ar[d] & \sE(h) \ar[r] \ar[d] & \sI_{{Y^1}\mid S_{(1,1)}}(2h_1+h_2)  \ar[r]  \ar[d]_\cong& 0\\
0 \ar[r] & \sO_F(h_1) \ar[r] \ar[d] & \sI_{Y^1|F}(2h_1+h_2) \ar[r] \ar[d] & \sI_{{Y^1}\mid S_{(1,1)}}(2h_1+h_2)  \ar[r] & 0.\\
& 0 & 0
}
\end{equation}
From it, we see that the locus where the two sections $s_1$ and $s_2$ are not independent is the surface $S_{(1,1)}$ whose defining equation induces the injective morphism from the lowest row of the previous diagram.

The points of $Y^1\cap Y^2$ are the points of the surface where the ideal sheaf $\sI_{Y^1\mid S_{(1,1)}}(2h_1+h_2)$ is not free. Since in any case  $\sI_{Y^1\mid S_{(1,1)}}(2h_1+h_2)$
 has depth two at any point of $S_{(1,1)}$, we see that $\sI_{Y^1\mid S_{(1,1)}}(2h_1+h_2)$ is a Cohen-Macaulay sheaf. Therefore, it should be free at any regular point of $S_{(1,1)}$. It follows that $Y^1\cap Y^2$ is contained in the singular locus of $S_{(1,1)}$. In particular, when $S_{(1,1)}$ is smooth,  $\sI_{Y^1\mid S_{(1,1)}}(2h_1+h_2) $ is a line bundle and $Y_1\cap Y_2=\emptyset$. The following proposition gives some information on the relation between the zero loci of the sections $s_i\in H^0(\sE(h_i))$ inside $S_{(1,1)}$.
 \end{proof}

\begin{proposition}\label{pEvaluation}
    Let $\sE$ be a special instanton bundle.  Then the ideal sheaves of the zero loci $Y^1,Y^2\subset S_{(1,1)}$ satisfy $\sI_{Y^1\mid S_{(1,1)}}(2h_1+h_2)\cong\sI_{Y^2\mid S_{(1,1)}}(h_1+2h_2)$. In particular, when $S_{(1,1)}$ is smooth, we have the linear equivalence of divisors $Y_1-Y_2\sim (h_1-h_2)_{\mid S_{(1,1)}}=-l+e_1+e_2+e_3$. In this case $\sE$ fits into the exact triple
\begin{equation}\label{sDelPezzoSections}
0 \to \sO_F(h_1) \oplus \sO_F(h_2) \to \sE(h) \to \sO_{S_{(1,1)}}\Bigl((4-k)l -2e_1-e_2+(k-1)e_3\Bigr)\to0
\end{equation} 
and we call this sequence the evaluation sequence of $\sE$.
\end{proposition}

\begin{proof}
Associated to the non-zero section $s_1\in \sE(h_1)$ we have the short exact sequence
$$
0\arr \sO_F(h_1)\oplus\sO_F(h_2)\arr \sE(h)\arr\sI_{{Y^1}\mid S_{(1,1)}}(2h_1+h_2)\arr 0.
$$
On the other hand, if we construct the same kind of exact sequence starting from $s_2\in H^0(\sE(h_2))$ we obtain
$$
0\rightarrow \sO_F(h_1)\oplus\sO_F(h_2)\rightarrow \sE(h)\rightarrow\sI_{{Y^2}\mid S_{(1,1)}}(h_1+2h_2)\rightarrow 0.
$$
Since the first maps on the two previous short exact sequences are the same, the cokernel sheaves are isomorphic. To conclude the proof observe that if $S_{(1,1)}$ is smooth then the linear equivalence of divisors follows from Lemma \ref{lRestrictionPic}. Moreover in this case $\sI_{{Y^1}\mid S_{(1,1)}}(2h_1+h_2)$ is a line bundle $\sF=\sO_{S_{(1,1)}}(L)$ for some $L \in \Pic(S_{(1,1)})$. In order to compute $\sO_{S_{(1,1)}}(L)$ explicitly, we recall that since $\sE$ is a charge $k$ instanton bundle, we have
\begin{equation}\label{chiconditions}
\begin{sistema}
\chi(\sE(-h))=0,\\
\chi(\sE)=2-2k, \\
\chi(\sE(-h_i))=-k.
\end{sistema}
\end{equation}
From the first two equations we obtain $\chi \bigl(\sO_{S_{(1,1)}}(L)\otimes \sO_F(-2h) \bigr)= \chi \bigl(\sO_{S_{(1,1)}}(L)\otimes \sO_F(-h)\bigr) +2k-2$. Recall that for a line bundle $\sO_{S_{(1,1)}}(D)$ the Riemann--Roch formula yields
\begin{equation}\label{rrlinebundle}
\chi\bigl(\sO_{S_{(1,1)}}(D)\bigr)=1+\frac{1}{2}D\bigl(D-K_{S_{(1,1)}}\bigr).
\end{equation} 
In particular, the first two equations of \eqref{chiconditions} give us $L H_{S_{(1,1)}} = 8-2k$. Let us denote $D_i$, $i=1,2$ the restriction of the divisor $2h_i+h_j$ on $S_{(1,1)}$. By Lemma \ref{lRestrictionPic}, the class of $D_i$ inside $\Pic(S_{(1,1)})$ is given by $D_i = (3+i)l-i(e_1+e_2+e_3)$; this implies that $D_i^2=13$ and $D_iH_{S_{(1,1)}}=9$. Consider now the last equation in \eqref{chiconditions}, which implies that $\chi \bigl(\sO_{S_{(1,1)}}(L-D_1) \bigr)= \chi \bigl(\sO_{S_{(1,1)}}(L-D_2)\bigr)$ and, by \eqref{rrlinebundle}, we obtain $\mathcal{L}(D_2-D_1)=0$. Denoting $\sO_{S_{(1,1)}}(L)=\sO_{S_{(1,1)}}(al-b_1e_1-b_2e_2-b_3e_3)$, we have
\begin{equation}
\begin{sistema}
L H_{S_{(1,1)}}=3a-b_1-b_2-b_3=8-2k,\\
L(D_2-D_1)=a-b_1-b_2-b_3=0.
\end{sistema}
\end{equation}
Thus we get $a=b_1+b_2+b_3=4-k$. Similarly, by \eqref{rrlinebundle} and $\chi \bigl(\sO_{S_{(1,1)}}(L-H_{S_{(1,1)}})\bigr) =2-2k$, we obtain $b_1^2+b_2^2+b_3^2=k^2-2k+6$.\\ 
To summarize, $\sO_{S_{(1,1)}}(L)=\sO_{S_{(1,1)}}(al-b_1e_1-b_2e_2-b_3e_3)$ satisfies
\[
\begin{sistema}
a=4-k,\\
b_1+b_2+b_3=4-k,\\
b_1^2+b_2^2+b_3^2=k^2-2k+6,
\end{sistema}
\]
and the divisor $L:=(4-k)l -2e_1-e_2+(k-1)e_3$ is the unique divisor (up to permutation on the coefficients of $e_i$) which satisfies the previous conditions.
\end{proof}
 
\begin{remark}\label{rRelationSections}
 Let $Y^{i}$ be the zero locus of $s_i\in H^0(\sE(h_i))$. Consider the short exact sequences
 \begin{gather}\label{etwistedsection}
 0\rightarrow \sO_F(-3h_1)\rightarrow\sE(-2h_1)\rightarrow\sI_{Y^1|F}(-h_1)\rightarrow 0,\\
 0 \rightarrow \sI_{Y^1|F}(-h_1) \rightarrow \sO_F(-h_1) \rightarrow \sO_{Y^1}(-h_1)\rightarrow 0. \notag
 \end{gather}
 In the proof of Proposition \ref{pHilbertComponents} we noticed that $h^0(\sO_{Y^{i}}(-h_i))=\ell(\bar{a_i})$ represent the numbers of lines (counted with multiplicities) appearing in $Y^{i}$. Taking the cohomology of the sequences \eqref{etwistedsection} we get $h^1(\sE(-2h_1))=h^2(\sE(-2h_1))=\ell(\bar{a_1})$. However the analogs of sequences \eqref{etwistedsection} for $s_2 \in H^0(\sE(h_2))$ yield $h^1(\sE(-2h_2))=h^2(\sE(-2h_2))=\ell(\bar{a_2})$. Serre's duality gives us $h^t(\sE(-2h_2))=h^{3-t}(\sE(-2h_2))$, thus $\ell(\bar{a_1})=\ell(\bar{a_2})$, i.e. $Y^1$ and $Y^2$ contain the same number of lines (counted with multiplicities).
\end{remark}

Let us make now some observations about thick structures. Notice that a reducible sextic  $S_{(1,1)}$ with singular locus a reducible conic $L_1\cup L_2$ contains the first infinitesimal neighborhoods $L_i^{(1)}$. However, we have the following result.
\begin{lemma}
    Let $Y$ be a complete intersection multiple structure supported on a line $L$ and contained in an arbitrary surface of type $S_{(1,1)}$. Then $Y$ is primitive.
\end{lemma}
\begin{proof}
Let us observe that the result is trivial when $S_{(1,1)}$ is smooth. Otherwise, suppose that $Y$ is a complete intersection structure supported on $L\in h_i^2$. Hence $\sI_{Y|F}$ has the following resolution
$$
0\to \sO_F(-(a+b)h_i)\to \sO_F(-ah_i)\oplus \sO_F(-bh_i) \to \sI_{Y|F}\to 0.
$$
with $1\leq a\leq b$. Tensoring the above sequence by $\sO_F(h)$, we get $h^0(\sI_{Y|F}(h))=0$, unless $a=1$. Therefore $Y$ is contained in a (smooth) surface $S_{(2-i,i-1)}$ and in particular is primitive.
\end{proof}

The previous Lemma implies that thick complete intersection structures supported on a line are excluded as components of zero loci of sections of special instanton bundles. In particular, we will show in Theorem \ref{tClassificationSpecials} that only two kinds of curves can actually occur. First we start with the following lemma. 
\begin{lemma}\label{lNoMultiple}
Let $C:=C^{i}_{a_t}$ be a rational curve of degree $a_t+1$ in the class $h_i^2+a_th_j^2$. Let $Y$ be a primitive extension of type $\sO_C$ on $C$. Then
\begin{itemize}
    \item [(i)] if $Y$ is contained in an irreducible del Pezzo surface $S_{(1,1)}$ then $a_t=1$, i.e. $C$ is a conic;
    \item [(ii)] if $Y$ is contained in a reducible del Pezzo surface $S_{(1,0)}\cup S_{(0,1)}$ then $a_t=0$, i.e. $C$ is a line.
\end{itemize}
\end{lemma}
\begin{proof} Let us start with point $(i)$. Suppose that $S_{(1,1)}$ is irreducible and smooth. Then the only curves as in the statement that can be contained in $S_{(1,1)}$ are lines, conics and cubics. If $Y$ is a primitive extension of multiplicity $k$ over a line, we can assume that its reduced structure is the exceptional divisor $e_1$ and therefore $Y$ is in the class $ke_1$. But the canonical divisor of such a curve would have degree $$(K_{S_{(1,1)}}+ke_1)ke_1=-(k+1)k$$ which is incompatible with the conditions of Theorem \ref{theor-equiv}, unless $k=1$ since $\omega_Y$ should have degree $-2k$. Analogously if $Y$ is a primitive extension of multiplicity $k$ over a cubic, then, thanks to Lemma \ref{lRestrictionPic}, its class is given by $k(2l-e_1-e_2-e_3)$. The canonical divisor of such a curve would have degree $$
\bigl(K_{S_{(1,1)}}+k(2l-e_1-e_2-e_3)\bigr)(2kl-ke_1-ke_2-ke_3)=k(k-3)
$$ 
which again is incompatible with the conditions of Theorem \ref{theor-equiv}, unless $k=1$.\\ 
We are left with the case of an irreducible $S_{(1,1)}$ having an $A_1$-type singularity or an $A_2$-type singularity and we work over the resolution of the singularity $S'$ of $S_{(1,1)}$. If $Y$ is a primitive extension of multiplicity $k$ on a line, its reduced structure is either the exceptional divisor $f$ or $g$. In both cases $\omega_Y$ has degree $-k(k+1)$, which is admissible if and only if $k=1$. Analogously, a primitive extension $Y$ of multiplicity $k$ over a cubic belongs to the class$k(2l-e-2f-g)$ in the Chow ring. The canonical divisor of this curve has degree $k^2-3k$, which again is always different from $-2k$ unless $k=1$. The case of an $A_2$-type singularity is completely analogous.

Now we deal with point $(ii)$. Let us start by noticing that $C$ is a complete intersection of type $h_j(h_i+(a_t-1)h_j)$, thus if $C^i_{a_t}$ is contained in a reducible del Pezzo $S_{(1,1)}$, then $C^i_{a_t} \subset S_{(0,1)}$ when $a_t \ge 1$. Using the same argument as in Lemma \ref{lRestrictionPic}, $S_{(0,1)}$ can be identified with the Hirzebruch surface $\mathbb{F}^1$ via the projection $F \xrightarrow{\pi_j} \pp^2$. Thus $C^i_{a_t}$ belongs to the linear system $C_0+a_tf$, and $Y$ is in the class $k(C_0+a_tf)$. The canonical divisor of this multiple curve would have degree $$\bigl(K_{S_{(0,1)}}+k(C_0+a_tf)\bigr)(kC_0+ka_tf)=(2a_t-1)k^2-(2a_t+1)k$$ which again contradicts Theorem \ref{theor-equiv} unless $k=1$.
\end{proof}
The previous lemma implies that the only multiple structures that can appear in the zero locus of a section of a special instanton bundle are
\begin{itemize}
    \item primitive extensions on conics contained in an irreducible del Pezzo surfaces $S_{(1,1)}$;
    \item primitive extensions on lines contained in a reducible del Pezzo surfaces $S_{(1,1)}=S_{(1,0)} \cup S_{(0,1)}$.
\end{itemize} 

In the next theorem we completely classify all the possible configurations of curves and del Pezzo surfaces associated to a special instanton bundle.
\begin{theorem}\label{tClassificationSpecials}
Let $\sE$ be a special instanton bundle of charge $k\ge 2$. Let $Y_{\bar{a}}^{i}$ be the zero locus of a section $s_i \in H^0(\sE(h_i))$. Then, only the following cases occur:
\begin{itemize}
    \item [(i)] $\ell(\bar{a})=1$ and the dependence sextic $S_{(1,1)}$ is irreducible, and is either smooth or has an $A_1$-type singularity;
    \item [(ii)] $\ell(\bar{a})=k$ and the dependence sextic $S_{(1,1)}$ is the reducible union $S_{(1,0)}\cup S_{(0,1)}$.
\end{itemize}
Moreover, in case $(i)$ the line contained in $Y^{i}$ is the pre-image of one blown-up point via $\pi_i$ and the conics in $Y^{i}_{red}$ are the strict transforms of lines passing through other (possibly double) blown-up points. In case $(ii)$ the lines lie in the ruling of $S_{(2-i,i-1)}$.
\end{theorem}
\begin{proof}
First of all notice that it is enough to prove the assertion when $Y^{i}$ is reduced.

Let us first consider the reducible case $(ii)$. Suppose $i=1$, the other case being analogous by symmetry. We have $S_{(1,1)}=S_{(1,0)} \cup S_{(0,1)}$. Since $Y \subset S_{(1,1)}$, each line appearing in the reduced structure of $Y$ would be a fiber in $S_{(1,0)}$, thus it will meet any other rational curve of higher degree in $S_{(1,0)}$. The only available option is to pick one rational curve in $S_{(0,1)}$ which does not meet the chosen fibers of $S_{(1,0)}$. Notice that any two rational curves of degree greater than one on $S_{(0,1)}$ meet each other, thus we can only have one of such curves as a component of $Y$. Moreover we cannot have primitive extensions on these curves thanks to Lemma \ref{lNoMultiple}.

Let us deal now with case $(i)$. We first prove that the case of an irreducible $S_{(1,1)}$ with an $A_2$-type singularity cannot occur. Using the notation of Lemma \ref{lSingDelPezzo}, $S$ contains exactly two lines, namely the images under $\tau$ of $g$ and $l-e-f-2g$ but, since $g(l-e-f-2g)=1$, they intersect. Thus the only possibility is to have the disjoint union of a single line and possibly multiple conics. Each conic lies in the linear system $|l-g|$, but any element of this system has positive intersection with the two lines, therefore this case can be excluded. Thus $S_{(1,1)}$ is either smooth or has an $A_1$ singularity. If $S_{(1,1)}$ is smooth, then it cannot contain primitive multiple lines by Lemma \ref{lNoMultiple}, hence $\ell(\bar{a})\le 3$. However if $\ell(\bar{a})>1$ then $Y_{red}$ would contain a rational curve $Y_{a_t}$, but each of these curves meet the exceptional lines. The same argument proves the statement also in the case of a singular irreducible $S_{(1,1)}$. To finish the proof we notice that the single line $L$ in $Y^{(i)}$ has class $h_i^2$ in $A^2(F)$, thus projects to a point $p$ of $\pp^2$ via $\pi_i$, in particular is the pre-image of a blown-up point, while every conic disjoint from $L$ projects to a line in $\pp^2$ passing through a blown-up point different from $p$. In the case of a singular irreducible $S_{(1,1)}$, using the notation of Lemma \ref{lSingDelPezzo}, we only have a ruling of conics $l-g$ and each conic projects via $\pi_i$ to a line passing through the simple blown-up point. The only two lines in $S_{(1,1)}$ projecting to a point are the exceptional divisors $f$ and $g$, but $g$ intersects every conic in the ruling, thus the only option is $f$ and the proof is concluded.
\end{proof}

As a direct consequence, we obtain the following existence result.

\begin{corollary}\label{cExistenceSpecial}
For any $k\ge 1$ there exist $\mu$-stable, special instanton bundles of charge $k$.
\end{corollary}
\begin{proof}
It follows directly from Proposition \ref{serreexistenceflag} and Theorem \ref{tClassificationSpecials}, since we can find curves of the form $Y_{\bar{a}}$, with $\bar{a}=\{1,k\}$, contained in degree six del Pezzo surfaces.
\end{proof}

Since any special instanton bundle is a $h_i$-'t Hooft bundle, it is natural to ask whether the converse also holds. 
\begin{lemma}\label{lNotSpecial}
For $k \ge 2$ there exist $h_1$-'t Hooft bundles which are not $h_2$-'t Hooft.
\end{lemma}
\begin{proof}
Thanks to Proposition \ref{0locus}, Proposition \ref{serreexistenceflag} and Remark \ref{rdelPezzo}, it is enough to construct a curve realising a $h_1$-'t Hooft bundles which is not contained in any hyperplane section of $F$. In particular, we show that it is possible to construct $h_1$-'t Hooft bundles of charge $2$ which are not $h_2$-'t Hooft. Let us consider the curve
\[
Y=Y_{(1,1)}=C_1 \cup C_2
\]
given by the disjoint union of two conics. Let us consider the short exact sequence
\[
0 \to \sI_Y \to \sO_F \to \sO_Y \to 0
\]
tensored by $\sO_F(h)$. Since $h^0(\sO_Y(h))=6$, the sheaf $\sI_{Y|F}(h)$ has at least two independent global sections. We show that actually $h^0(\sI_{Y|F}(h))=2$. Since there exists a degree $6$ del Pezzo surface $S_{(1,1)}$ in $|\sO_F(h)|$ containing $Y$, the two conics must lie in the same ruling $|l-e_i|$. Suppose that both $C_i$ are in the linear system $|l-e_1|$. Consider the short exact sequence
\[
0 \to \sI_{S_{(1,1)}|F} \to \sI_{Y|F}\to \sI_{Y|S_{(1,1)}} \to 0,
\]
and twist it by $\sO_F(h)$ obtaining
\[
0 \to \sO_F \to \sI_{Y|F}(h) \to \sO_{S_{(1,1)}}(l+e_1-e_2-e_3)\to 0.
\]
In particular 
\[
h^0(\sI_{Y|F}(h))=h^0(\sO_F)+h^0\bigl(\sO_{S_{(1,1)}}(l+e_1-e_2-e_3)\bigr)=2,
\]
thus there exists a $\pp^1$ of degree six del Pezzo surfaces containing $Y$. Since the lines of the family $h_1^2$ not intersecting $Y$ move in an open set of a $\pp^2$, it is possible to choose $L_1$ not lying in any $S_{(1,1)}$ containing $Y$. Finally the scheme $Y'=Y \cup L_1$ gives a $h_1$-'t Hooft bundle which is not $h_2$-'t Hooft.
\end{proof}

Thus not every $h_i$-'t Hooft is a special instanton. However the case $k=1$ has a unique behaviour as we see in the next remark.

\begin{remark}
\label{rContainedDelPezzo}
Notice that from the previous lemma we see that the union of two disjoint conics $C_1\cup C_2$ as well as the union $C_1\cup L$ of conic  with a disjoint line are contained in infinitely many del Pezzo surfaces $S_{(1,1)}$. On the other hand the disjoint union $C_1\cup\dots \cup C_n \cup L$ of $n$ smooth conics and one line, for $n\geq 2$, is in general not contained in any del Pezzo surface $S_{(1,1)}$ and in the case it is indeed included in $S_{(1,1)}$, the curve uniquely determines the surface $S_{(1,1)}$. 

In the case of reducible del Pezzo surfaces, using a similar argument, one finds that two lines $L_1$ and $L_2$ from the same family are always contained in a smooth cubic surface $S_{(1,0)}$ or $S_{(0,1)}$. However, as soon as we consider three lines, then, in general, the union is not contained in any cubic surface. In case it is actually contained, the curves determine the cubic.
\end{remark}

In light of Proposition \ref{dep}, we take a step forward and conclude this section describing the vanishing locus of a section in $H^0(\sE(h))$, of a special instanton bundle $\sE$, obtained by combining the elements of $H^0(\sE(h_i))$.
This result will be used when describing the restriction of instanton bundles to conics (see Section \ref{sJumpingConics}).

\begin{proposition}\label{prop-geometriaE11}
     Let $\sE$ be a special instanton bundle of charge $k\ge 2$ and take the unique sections $s_1 \in H^0(\sE(h_1))$ and $s_2 \in H^0(\sE(h_2))$. Let $S_{(1,1)}\subset F$ be the sextic surface obtained as the degeneration locus of the sections $s_1,s_2$. Let $\Gamma\subset F$ be a conic and let $t_i\in H^0(\sO_F(h_i))$ be two global sections that define $\Gamma$.
    Then, for any $(\alpha: \beta)\in \mathbb{P}^1$, with $\alpha\beta \neq 0$, the vanishing locus of $\bar{s} = \alpha s_1 t_2 + \beta s_2 t_1 \in H^0(\sE(h))$ is a curve $\Delta \subset F$ representing the class $(k+3)h_1h_2$ such that
    \begin{itemize}
        \item If $\Gamma \not\subset S_{(1,1)}$, then $\Delta = \Gamma \cup \Upsilon$, where $\Upsilon \subset S_{(1,1)}$ and $\rm{length}(\Gamma\cap \Upsilon)=2$.
        \item If $\Gamma\subset S_{(1,1)}$, then $\Delta = \tilde{\Gamma} \cup \Upsilon$, where $\tilde{\Gamma}$ is a double structure on $\Gamma$ such that $\tilde{\Gamma}\not\subset S{(1,1)}$ and $\Upsilon\subset S_{(1,1)}$.
    \end{itemize}
    In any case, given a point $p \in S_{(1,1)}$, there exists a pair $(\alpha:\beta)$ 
    such that $p$ lies in $V(\bar{s})$.
\end{proposition}

\begin{proof}

Let $\Gamma\subset F$ be any conic and let $\Delta$ be the zero locus of the section $\bar{s}\in  H^0(\sE(h))$ constructed in the statement. If $\Delta$ had a codimension one component it would mean that $\Delta= Y_i \cup S_{(2-i,i-1)}$, situation easily excluded considering the particular form of $\bar{s}$. Therefore $\Delta$ is purely of codimension 2, $\Gamma\subset \Delta$ and  from the exact triple

$$
0\rightarrow \sO_F\rightarrow\sE(h)\rightarrow\sI_{\Delta|F}(2h)\rightarrow 0
$$
we see that $\Delta$ represents in the Chow ring the class $(k+3)h_1h_2$. It is also clear that $\Delta\not\subset S_{(1,1)}$ since otherwise $h^0(\sE)\neq 0$, contradicting the definition of an instanton bundle. On the other hand, the zero locus of $\bar{s}_{\mid S_{(1,1)}}$ is a curve representing the class $(k+2)h_1h_2$.

Consider the evaluation sequence \eqref{sDelPezzoSections}. It induces the short exact sequence
\[
0 \arr \mathcal{I}_\Gamma(h) \arr  \sI_\Delta (2h) \arr \sI_{{Y^1}\mid S_{(1,1)}}(2h_1+h_2) \arr 0.
\]
Localizing this exact sequence to any point $p\not\in S_{(1,1)}$, we see that the $\Delta_{red} \cap (F\setminus S_{(1,1)})\subset\Gamma\setminus S_{(1,1)}$.

Therefore, in case $\Gamma\not \subset S_{(1,1)}$, we see that  $\Delta = \Gamma \cup \Upsilon$, where $\Upsilon \subset S_{(1,1)}$. Moreover, applying the adjunction formula we see that $p_a(\Upsilon)=0$. Since the zero locus $\Delta$ of a global section of the bundle $\sE(h)$ should have $p_a(\Delta)=1$, we can conclude that
$$
\rm{length}(\Gamma\cap\Upsilon)=p_a(\Delta)-p_a(\Gamma)-p_a(\Upsilon)+1=2.
$$

On the other hand, if $\Gamma\subset S_{(1,1)}$, from the aforementioned restrictions, we see that $\Delta = \tilde{\Gamma} \cup \Upsilon$, where $\tilde{\Gamma}$ is a double structure on $\Gamma$ such that $\tilde{\Gamma}\not\subset S_{(1,1)}$ and $\Upsilon\subset S_{(1,1)}$.

In order to prove the last claim of the statement, fix a point $p \in S_{(1,1)}$ and consider a local description of the section $\bar{s}$ in an open neighbourhood of the point. Its evaluation at the point determines the pair $(\alpha:\beta)$ that defines the linear combination.


\end{proof}

\section{Moduli spaces of 't Hooft bundles}\label{sModulitHooft}

In this section we will describe the moduli spaces of 't Hooft bundles. We will say that a torsion free sheaf $\sE$ is an \textit{instanton sheaf} if and only if it satisfies all the conditions of Definition \ref{dInsta} but the local freeness. The first step is the following key correspondence.
\begin{proposition}\label{correspondence}
There exists a natural one to one correspondence between 
$$
\left\{ 
\begin{matrix}
\text{Special $\mu$-stable instanton} \\
\text{bundles of charge $k$}
\end{matrix}
\right\}
  \leftrightarrow
\left\{
\begin{matrix}
\text{ Curves $Y$ satisfying the conditions in} \\ 
\text{Theorem \ref{theor-equiv} and a generating} \\ 
\text{section of $\wedge^2 (\sN_Y) \otimes \sO_Y(-2h_i)$} 
\end{matrix}
\right\}.
$$
Any special instanton bundle is uniquely determined by
\begin{itemize}
\item [a)] a del Pezzo surface $S_{(1,1)}$ of degree $6$ in $F$ without $A_2$ singularity;
\item [b)] a choice of a pair $(R,\sC)$ where:
\begin{itemize}
    \item in the irreducible case $R$ is a ruling of conics and $\sC$ is a line inside $S_{(1,1)}$; 
    \item in the reducible case $R$ is the ruling of lines inside $S_{(2-i,i-1)}$ and $\sC$ is a rational curve of degree $\ell(\bar{a})+1$ inside $S_{(i-1,2-i)}$ not intersecting any curve in $R$;
\end{itemize}
\item [c)] an element $\xi$ in the linear system $|kR|$;
\item [d)] a generating section of $\wedge^2 (\sN_Y ) \otimes \sO_Y(-2h_i)$ where $Y={\xi \cup \sC}$.
\end{itemize}
Conversely, any such data (a), (b), (c), (d) arise from a unique special instanton bundle.
\end{proposition}
\begin{proof}
The correspondence itself is a direct consequence of Serre's correspondence (cf. \cite[Theorem 1]{Ar}) and the results obtained in Lemma \ref{lIndices}, Proposition \ref{serreexistenceflag} and Theorem \ref{tClassificationSpecials}. Thus, to any special instanton bundle we associate the zero--locus of the section $s_i \in H^0(\sE(h_i))$, plus a generating section of $\wedge^2 (\sN_Y ) \otimes \sO_Y(-2h_i)$. Conversely, given a curve $Y$ as in Proposition \ref{0locus} and a generating section of $\wedge^2 (\sN_Y ) \otimes \sO_Y(-2h_i)$, we consider the corresponding element in $\Ext^1(\sI_{Y}(2h_i),\sO_F)$, which gives us a unique $\mu$-stable instanton bundle (up to isomorphism) thanks to Proposition \ref{serreexistenceflag}.
\end{proof}
Before dealing with the main theorems of this section, we prove the following preliminary result.
\begin{lemma}\label{lRuling}
Let $S_{(1,1)}$ be a smooth del Pezzo surface of degree $6$ and let $\xi$ be an element of the linear system $|k(l-e_i)|$ with $i\in\{1,2,3\}$, then $\xi$ is either the disjoint union of (possibly multiple) conics or the union of the two special curves $k(l-e_i-e_j)$ and $ke_j$ for $j\neq i$ .
\end{lemma}
\begin{proof}
We start by observing that the linear system $|ke_j|$ (and $|k(l-e_i-e_j)|$) have projective dimension zero and the only effective divisor in it corresponds to a non-reduced curve of degree $k$ and arithmetic genus $g=-\frac{k(k+1)}{2}+1$. Obviously the union of these two divisors belongs to the linear system $|k(l-e_i)|$.

To complete the proof we compute the dimension of the linear system $|k(l-e_i)|$. On the one hand we get
$$
h^0(k(l-e_i))=\binom{k+2}{2}-\binom{k+1}{2}=k+1.
$$
On the other hand, considering the blowup map $\sigma:S_{(1,1)} \to \pp^2$, the inverse image of any line passing through the blown-up point $\sigma(e_i)=p_i$ is an element of the linear system $|l-e_i|$.
The result directly follows by noticing that the space of $k$ lines passing through $p_i$ has affine dimension $k+1$.
\end{proof}

We are now ready to describe the moduli spaces of these vector bundles. In the case $k=1$ all instanton bundles are special thanks to Remark \ref{rContainedDelPezzo}, thus we refer to \cite[Theorem 1.3]{MMP}.

\begin{theorem}\label{tModuliSpecial}
The moduli space ${MI_s}(k)$ of $\mu$-stable, special instanton bundles of charge $k\geq 2$ consists of two irreducible, smooth components ${MI_s'}(k)$ and ${MI_s''}(k)$ of dimension $7+2k$ and $4k+4$, respectively.
\end{theorem}

\begin{proof}
Thanks to Proposition \ref{correspondence} it is enough to describe the variety of moduli ${\overline{MI}_s}(k)$ of stable instanton sheaves determined by conditions $a)$, $b)$, $c)$ and taking in point $d)$ a (non necessarily generating) section of $\wedge^2 (\sN_Y ) \otimes \sO_Y(-2h_i)$. The space ${MI_s}(k)$ will be an open subset of ${\overline{MI}_s}(k)$. The variety $\overline{MI}_s(k)$ is fibered over a variety $M$ by  
$$\mathbb{P}\bigl(H^0(\wedge^2 (\sN_Y ) \otimes \sO_F(-2h_i)_{|Y})\bigr) \cong \pp\bigl(\Ext^1(\sI_Y(2h_i),\sO_F)\bigr)\cong \mathbb{P}^k.$$
Now we describe $M$. It is fibered over the subset of $\pee7$ consisting of irreducible del Pezzo sextics which do not have an $A_2$ singularity (see Theorem \ref{tClassificationSpecials}), i.e. $\Lambda:=\Lambda_{sm} \cup \Lambda_{A_1} \cup \Lambda_r$. Thus we have the following situation:
\[
{\overline{MI}_s(k)} \overset{\Gamma}{\twoheadrightarrow} M  \overset{\Psi}{\twoheadrightarrow} \Lambda
\]
The fibers of $\Psi$ consists of an open subset of the following spaces:
\begin{itemize}
    \item six copies of $\mathbb{P}\bigl(H^0(\sO_{S_{(1,1)}}(kR)\bigr) \cong \mathbb{P}^k$ if the dependence sextic is smooth;
    \item a single copy of $\mathbb{P}^k$ if the dependence sextic is irreducible with an $A_1$-type singularity;
    \item the product $\mathbb{P}^k\times \pp^{2k}$ if the dependence sextic is reducible.
\end{itemize}
In the smooth case, the six copies correspond to a choice of the pair $(R,\sC)$ in  $b)$. In the singular irreducible case the fibers represent the choice of $k$ conics in the ruling of proper transforms of lines passing through the single point. In the reducible case $\mathbb{P}^k\times \pp^{2k}$ corresponds to the choice of $k$ fibers in one cubic surface and a smooth, rational, complete intersection curve of degree $k$ on the other cubic surface.

We will now prove that $M$ has two connected components. In order to prove the connectedness of the two components we show that one can connect different choices in $b)$ in Proposition \ref{correspondence} by varying the smooth del Pezzo surface of degree $S_{(1,1)}$. The fact that $M=\tilde{M} \cup \hat{M}$ follows from Proposition \ref{pHilbertComponents} and Theorem \ref{tClassificationSpecials}, since it is not possible to deform two different choices in $a)$ and $b)$ starting with an irreducible $S_{(1,1)}$ and finishing with a reducible one.

\medbreak
\textbf{Connectedness and smoothness of $\tilde{M}$.\\}
$\tilde{M}$ is fibered over $\Lambda_{sm} \cup \Lambda_{A_1}$, which is an open subset of $\mathbb{P}^7$ by Lemma \ref{lOpenDelPezzo}. Now we will explicitly show that it is possible to connect any two different choices of line-conic configurations that give a special instanton bundle, where $S_{(1,1)}$ can be smooth or singular for either one of the chosen configurations. Furthermore, we will ensure that any sheaf associated to the points of the connecting path is again a special instanton bundle given by a line-conic configuration with the same Chern classes. To do so, we will divide the proof in two steps: we will first connect any two configurations which live in a smooth del Pezzo surface and then we will connect two configurations that live respectively in a smooth and a singular (of type $A_1$) irreducible del Pezzo surface. We use the notation introduced in Section \ref{sDescriptionDelPezzo}.

\smallskip
\textbf{Step 1: connecting configurations in smooth $S_{(1,1)}$}.\\
Given three non collinear points $Z=\{p_0,p_1,p_2\}$ in $\pp^2$, we can define their ideal in terms of the three lines that pass through any two of the points. Denoting their defining linear forms by $\ell_0, \ell_1$ and $\ell_2$ respectively, we have  
$$
I_Z = \langle \ell_0\ell_1, \ell_0\ell_2, \ell_1\ell_2 \rangle.
$$
This ideal can be described in a determinantal way considering the $2\times 2$ minors of the matrix
$$
\left(
\begin{array}{ccc}
\ell_0 & 0 & -\ell_2\\
0 & \ell_1 & \ell_2
\end{array}
\right).
$$
Observe that, since $\ell_0, \ell_1$ and $\ell_2$ are linearly independent, we can always transform the latter matrix to the one described in \eqref{mat-det-points}, using linear combinations of rows and columns.\\
Let us consider the curve $Y= \xi \cup l$ given by the choices in $a)$, $b)$ and $c)$ and let us push it forward to $\mathbb{P}^2$ via $\pi_1$. Since $\xi$ is different from the union of the two special curves described in Lemma \ref{lRuling}, the lemma itself allows us to specialise every such $\xi$ to a primitive multiple conic $\bar{C}$ of multiplicity $k$.\\
Thus let us consider $\xi$ to be a primitive extension of multiplicity $k$ supported on one conic, whose image on the projective plane to which we project is a line $L$ of multiplicity $k$ passing through one of the blown-up points but not through anyone of the other two. Without loss of generality, we suppose it to be $p_1= V(\ell_0,\ell_2)$  and $L = V(\ell_0 + \ell_2)$ (a line that we can consider reduced). 
We will now describe a closed path that permutes the two linear forms $\ell_1$ and $\ell_2$. In particular, this path will permute the points $p_1$ and $p_2$ and henceforth move the line $L$ to the line $L'=V(\ell_0,\ell_1)$, passing through $p_2$ but not through anyone of the other two blown-up points. Moreover, we would like to maintain, for any point of the path, the same geometrical configuration we just described: three non aligned points with a line passing only through one of them.

Consider the map
$$
\begin{array}{rccc}
g: & [0,1] &\longrightarrow & \mathbb{C}\\
& t & \mapsto & t e^{\pi i (1-t)}
\end{array}
$$
which allows us to define, for $t\in [0,1]$, three linear forms 
$$
\begin{array}{l}
\tilde{\ell}_{0,t} = \ell_0,\\
\tilde{\ell}_{1,t} = \bigl(1-g(t)\bigr)\ell_1 + g(t)\ell_2,\\
\tilde{\ell}_{2,t} = g(t)\ell_1 + \bigl(1-g(t)\bigr)\ell_2.
\end{array}
$$
Notice that, as wanted $\tilde{\ell}_{1,0}=\tilde{\ell}_{2,1}=\ell_1$ and $\tilde{\ell}_{2,0} = \tilde{\ell}_{1,1} = \ell_2$.\\
Described as elements of the vector space $H^0(\sO_{\pp^2}(1))$ with chosen basis $\ell_0$, $\ell_1$, and $\ell_2$ the three new linear forms are represented by the matrix
$$
A_t=
\left(
\begin{array}{ccc}
1 & 0 & 0\\
0 & 1-g(t) & g(t) \\
0 & g(t) & 1-g(t)
\end{array}
\right).
$$
Consider its determinant 
$$
h(t)= \det(A_t) = 1-2t e^{\pi i (1-t)}
$$
and notice that $h(t) \neq 0$ for any $t\in [0,1]$. Indeed, a direct computation shows that $h(0)=1, h(1)=-1$ and $h(t) \in \mathbb{C} \backslash \mathbb{R}$ for any $t \in (0,1)$. This means that the linear forms  $\tilde{\ell}_{0,t}$, $\tilde{\ell}_{1,t}$, $\tilde{\ell}_{2,t}$ are linearly independent at any point of the path.\\ 
Finally, consider the line $L_t = V(\tilde{\ell}_{0,t}+\tilde{\ell}_{2,t})$. The defined line obviously contains the point $p_{1,t}= V(\tilde{\ell}_{0,t},\tilde{\ell}_{2,t})$; moreover, since the linear forms are independent, the line contains neither $p_{0,t}= V(\tilde{\ell}_{1,t},\tilde{\ell}_{2,t})$ nor $p_{2,t}= V(\tilde{\ell}_{0,t},\tilde{\ell}_{1,t})$, for any $t \in [0,1]$.

To conclude, notice that the described path changes the choice of the ruling and fixes the line. By a completely similar argument, it is possible to connect two different choices for the line with the ruling fixed.

\smallskip
\textbf{Step 2: connecting configurations, respectively, in a smooth and singular $S_{(1,1)}$.}\\
Due to the previous description of the del Pezzo surfaces (see Theorem \ref{tClassificationSpecials}) and to simplify computations, we consider $S_{(1,1)}$ irreducible with an $A_1$-type singularity. To construct it we have blown up the points $(1:0:0)$ and $(0:1:0)$ of the projective plane and a point on the exceptional divisor over $(1:0:0)$. This means that, once the curve $Y^1$ is projected  on $\pp^2$, we have $k$ (possibly multiple) lines, representing the projections of the conics in $Y^1$ that pass through $(0:1:0)$ but not through $(1:0:0)$. The point $(1:0:0)$ is the projection of the line in $Y^1$. Under these assumptions, we can define the lines by
\begin{equation}\label{linesY}
\alpha_i x + \beta_i z = 0, \mbox{ for } i=1,\ldots,k,
\end{equation}
Consider a family of surfaces $S_{(1,1)}$ parameterized by 
$$
\left(
\begin{array}{ccc}
x_0 & x_1 & x_2\\
\lambda x_0 + a_{2,0} x_2 & x_1 + a_{2,1} x_2 & 0
\end{array}
\right),
$$
with fixed $a_{2,0}, a_{2,1}$ and denoting by $\lambda$ the parameter. Observe that, if $\lambda=0$, we get the required irreducible singular surface, while, for any $\lambda \neq 0$, we get a smooth del Pezzo surface constructed by blowing up the non-aligned points
$$
(1:0:0), \quad (0:1:0) \quad \text{and} \quad (-a_{2,0}:a_{2,1}:\lambda).
$$
Varying the parameter $\lambda$ slightly in order to assure that the point $(-a_{2,0}:-a_{2,1}:\lambda)$ is not contained in the lines defined in (\ref{linesY}), we get the required connecting path.

Combining the described paths, we can connect any two choices in $b)$, thus the variety $\tilde{M}$ is connected. The dimension count follows directly from the previous description. Finally we deal with the smoothness of $\tilde{M}$. Consider the projection map
\[
\tilde{M}\arr H \subset \mathcal{H}:=\mathrm{Hilb}^{(2k+1)t+k+1}(F)
\]
which projects an element of $\tilde{M}$ to the associated curve in $\sH$. Recall that $H$ is the open subset of $\sH$ of curves satisfying the conditions of Theorem \ref{theor-equiv}. Since any point in $H$ representing a curve is contained in at most one del Pezzo surface of degree six, thanks to Remark \ref{rContainedDelPezzo}, $\tilde{M}$ projects isomorphically onto a component of $H$. The smoothness of $\tilde{M}$ follows from \eqref{eCohomologyNormal}, since any point of $H$ is a smooth point in the Hilbert scheme $\sH$ and $H$ is an open dense subset of $\sH$.

\medbreak
\textbf{Connectedness and smoothness of $\hat{M}$}.\\
Consider a curve $Y$ as in Proposition \ref{correspondence}. Notice that 
$$
\hat{M} =\{(Y,S_{(1,0)} \cup S_{(0,1)}) \ | \ Y \subset S_{(1,0)} \cup S_{(0,1)}\}  \xrightarrow{p} \pp^2 \times \pp^{\vee 2}
$$
is an incidence variety fibered over $\pp^2 \times \pp^{\vee 2}$ which parameterizes the reducible del Pezzo surfaces of degree six in $F$. The fibers of $p$ are all isomorphic to an open subset of the product variety of $\pp\bigl(H^0(\sO_{S_{(1,0)}}(kf))\bigr)\cong \pp^{k}$ and $\pp\bigl(H^0(\sO_{S_{(1,0)}}(C_0+kf))\bigr) \cong \pp^{2k}$. All the fibers are smooth, connected and irreducible, the map $p$ is flat so we finally conclude that $\hat{M}$ is smooth and irreducible of dimension $3k+4$. Thus the proof is complete.
\end{proof}

As a quite straightforward by-product of the previous results, we  obtain a description of the structure of the moduli space of $h_i$-'t Hooft instanton bundles which will complete this section.
\begin{theorem}\label{tModuliHooft}
The moduli space ${MI^i}(k)$ of $\mu$-stable $h_i$-'t Hooft instanton bundles of charge $k\geq 2$ is a smooth variety consisting of at least $k$ irreducible components of dimension $5k+2$. Moreover the total space ${MI_H}(k):= MI^1(k)\cup MI^2(k)$ is singular along the two smooth components of $MI_s(k)$.
\end{theorem}

\begin{proof}
As in the proof of Theorem \ref{tModuliSpecial}, we will describe the moduli space $\overline{MI^i}(k)$ of stable $h_i$-'t Hooft instanton sheaves of charge $k$ arising from the curves described in Theorem \ref{theor-equiv} by taking any (non-necessarily generating) section of $\wedge^2 (\sN_Y ) \otimes \sO_Y(-2h_i)$. The space ${MI^i}(k)$ will be open inside $\overline{MI^i}(k)$. First of all notice that as an immediate consequence of Lemma \ref{lNotSpecial}, we obtain that for any $k\geq 2$ the moduli spaces $\overline{MI}^1(k)$ and $\overline{MI}^2(k)$ are two distinct varieties intersecting along the moduli space of special instantons. 

We start proving that $\overline{MI}^i(k)$ is smooth. In order to do that, notice that $\overline{MI}^i(k)$ is given by the choice of an element of the open subset $H\subset\mathcal{H}:=\mathrm{Hilb}^{(2k+1)t+(k+1)}(F)$ of
the Hilbert scheme of curves $Y$ of degree $2k+1$ and arithmetic genus $p_a(Y)=-k$ satisfying the condition of Theorem \ref{theor-equiv}. The smoothness of $H$ follows directly from \eqref{eCohomologyNormal}. We will briefly recall the various possibilities for the reader convenience.

Since $Y$ is the disjoint union of multiple structures supported on smooth rational curves, we can deal with the normal bundle of each component separately (see \eqref{eNormal2} and \eqref{eNormal3}):
\begin{itemize}
\item for a line (either simple or multiple) $L\subset F$, we have $\mathcal{N}_{L\mid F}\cong \sO_{\pp^1}^{\oplus 2}$;
\item for a smooth, rational, complete intersection curve $C\subset F$ as in Proposition \ref{com-int}, we have $\mathcal{N}_{C\mid F}\cong \sO_{\pp^1}(1)\oplus \sO_{\pp^1}(2a-1)$ with $a \ge 1$;
\item for a primitive multiple curve $D$ with support a smooth, rational, complete intersection curve $C\subset F$ as in Proposition \ref{com-int}, $\mathcal{N}_{C\mid F}\cong \sO_{\pp^1} \oplus\sO_{\pp^1}(2a)$.
\end{itemize}
In any case, $h^1(\mathcal{N}_{Y\mid F})=0$. Therefore, $\overline{MI}^i(k)$ is a fibration $\overline{MI}^i(k)\xrightarrow{p} H$ over the smooth base $H$. We claim that it is a smooth variety. Indeed, $p$ is a flat and surjective morphism on $H$ and the fibers of $p$ are represented by
\[
\pp \Bigl(H^0(\wedge^2 (\sN_Y ) \otimes \sO_F(-2h_i)_{|Y})\Bigr)\cong \pp\Bigl(\Ext(\sI_Y(2h_i),\sO_F)\Bigr)\cong \mathbb{P}^k,
\]
which is smooth and irreducible, thus the same holds for $\overline{MI}^i(k)$ and ${MI}^i(k)$.

To prove that $MI^i(k)$ consists of at least k irreducible components, it is enough to observe that, thanks to the previous description, $MI^i(k)$ is irreducible when restricted to each irreducible component of $H$, thus the result follows from Proposition \ref{pHilbertComponents}. 

Let us denote by $MI^i_{\ell(\bar{a})}(k)$ the restriction of $MI^i(k)$ to the pre-image of $H_{\ell(\bar{a})}$. From Theorem \ref{tClassificationSpecials}, we deduce that $MI^1_{\ell(\bar{a})}(k)$ and $MI^2_{\ell(\bar{a})}(k)$ are disjoint if and only if $\ell(\bar{a})\neq 1,k$ and if $\ell(\bar{a})$ is either $1$ or $k$, $MI^1_{\ell(\bar{a})}(k)$ and $MI^2_{\ell(\bar{a})}(k)$ intersect along the smooth  irreducible varieties $MI'_s(k)$ and $MI''_s(k)$, thanks to Theorem \ref{tModuliSpecial}. Finally, the assertion on the dimension of $MI^i(k)$ follows directly from the previous description of the moduli space as a fibration, since $H$ has dimension $4k+2$ thanks to \eqref{eCohomologyNormal} and each fiber has dimension $k$.\end{proof}
We end this section with a remark about the case $k=2$.
\begin{remark}
In the case $k=2$ we have quite an interesting picture. Thanks to Theorem \ref{tModuliSpecial}, the moduli space of special instanton bundles has two irreducible components $MI'_s(2)$ and $M''_s(2)$ of dimension $11$ and $12$ respectively. However, Theorem \ref{tModuliHooft} implies that the total moduli space of 't Hoof bundles $MI_H(2)$ has two irreducible components $MI'_H(2)$ and $MI''_H(2)$ of dimension 12. The first one is singular along the loci $MI'_s(2)$, while the second one coincides with $M''_s(2)$, thus is smooth. To see that $MI''_H(2)\cong M''_s(2)$ it is enough to observe that any configuration of two lines (or a double line) and a twisted cubic is always contained in a reducible del Pezzo surface $S=S_{(1,0)}\cup S_{(0,1)}$.  However the picture is different for $k\ge3$, since $k$ generic lines are not contained in any cubic surface $S_{(1,0)}$, thus special instanton bundles do not cover any irreducible component of $MI_H(k)$.
\end{remark}

  \section{The splitting type of an instanton bundle on conics}\label{sJumpingConics}
In this section we describe the behaviour of 't Hooft bundles when restricted to conics. We start with a general result on the bound of the possible splitting type of an instanton bundle when restricted to a conic. Let us recall the monadic description of an instanton bundle. 
\begin{theorem}\cite[Theorem 5.2]{MMP}\label{firstmonad}
Let $\sE$ be an instanton bundle of charge $k$ on $F \subset \pp(V_1)\times \pp(V_2)$. Then, up to permutation, $\sE$ is the cohomology of a monad 
\begin{equation}\label{mon}0 \to \sO_F(-h_1)^{\oplus k}\oplus \sO_F(-h_2)^{\oplus k} \xrightarrow{\alpha}  \sO_F^{\oplus 4k+2}
\xrightarrow{\beta} \sO_F(h_1)^{\oplus k}\oplus \sO_F(h_2)^{\oplus k}\to 0.
\end{equation}

Moreover, the monad is self-dual, i.e. it is possible to find a non-degenerate symplectic form $q: W \rightarrow W^\vee$, with $W$ a $(4k+2)$-dimensional vector space describing the copies of the trivial bundle in the monad, such that $\beta = \alpha^\lor \circ (q \otimes id_{\sO_F})$.
Reciprocally, any vector bundle with no global sections defined as the cohomology of such a monad is a charge $k$ instanton bundle.
\end{theorem}

The previous monad can be rewritten in the following form
\begin{equation}\label{mon1}
0 \longrightarrow
\begin{matrix}
H_1 \otimes \sO_F(-h_1)\\
\oplus \\
H_2 \otimes \sO_F(-h_2)
\end{matrix}
\xrightarrow{A} W \otimes \sO_F
\overset{A^t J\ }{\longrightarrow}
\begin{matrix}
H_1^\lor \otimes \sO_F(h_1)\\
\oplus \\
H_2^\lor \otimes \sO_F(h_2)
\end{matrix}
\longrightarrow 0,
\end{equation}
where $H_1$, $H_2$ and $W$ are vector spaces of dimension $k$, $k$, and $4k+2$ respectively and $J$ is a non-degenerated skew-symmetric bilinear form $J:W\times W\rightarrow \mathbb{C}$. Recall that these vector spaces are obtained through a Beilinson complex constructed from the instanton bundle (see \cite[Section 5]{MMP} for more details).

Given a point $p \in F$, denote by $A(p)$ the evaluation of the matrix $A$, representing the morphism $\alpha$ of the monad, at the point $p$. We have a map $A(p): H_1 \oplus H_2 \rightarrow W$ whose image we will denote by $U_{p}:=A(p)(H_1 \oplus H_2)$. Observe that, since $A^t J A =0$, we have $U_p \subset U_{p}^{\circ}$ where $Z^{\circ}$ denotes the annihilator of a vector subspace $Z\subset W$ with respect to $J$. Given a point of the flag variety $p \in F$, the fiber of the instanton bundle $\sE$ at  $p$ is  $U_{p}^{\circ}/U_{p}$. Consider now the following display of the monad \eqref{mon}:

\begin{equation}\label{displaymonad}
\begin{array}{c}
0 \longrightarrow \sK \longrightarrow \sO_F^{\oplus 4k+2} \longrightarrow \sO_F(h_1)^{\oplus k}\oplus \sO_F(h_2)^{\oplus k} \longrightarrow 0 \vspace{0.5 cm}\\
0 \longrightarrow \sO_F(-h_1)^{\oplus k} \oplus \sO_F(-h_2)^{\oplus k} \longrightarrow \sK \longrightarrow \sE \longrightarrow 0.
\end{array}
\end{equation}

Once we have fixed this notation, arguing as in \cite{ADHM}, the monad allows us to give a bound on the splitting type of conics:

\begin{theorem}\label{splitting-type-from-monad}
Let $C=\langle p,q\rangle$ be the unique smooth conic defined by two non-aligned points $p,q \in F$. Then $\sE_{\mid C}\cong\sO_{\pp^1}(-s)\oplus\sO_{\pp^1}(s)$ if and only if $\dim(U^{\circ}_{p}\cap U_{q})=s$.
\end{theorem}

\begin{proof}
Let us restrict the display of the monad \eqref{displaymonad} to  the conic $C$
\begin{equation}\label{displaymonadconic}
\begin{array}{c}
0 \longrightarrow \sK_C \longrightarrow \sO_C^{\oplus 4k+2} \longrightarrow \sO_C(h_1)^{\oplus k}\oplus \sO_C(h_2)^{\oplus k}\cong\sO_{\pp^1}(1)^{\oplus 2k} \longrightarrow 0 \vspace{0.5 cm}\\
0 \longrightarrow \sO_C(-h_1)^{\oplus k} \oplus \sO_C(-h_2)^{\oplus k}\cong\sO_{\pp^1}(-1)^{\oplus 2k}  \longrightarrow \sK_C \longrightarrow \sE_C \longrightarrow 0
\end{array}
\end{equation}
and we denote by $A_C$ the $(4k+2)\times 2k$ matrix of linear polynomials on $C\cong\pp^1$ obtained by restricting the matrix $A$ from (\ref{mon1}) to $C$. 

We observe that $H^0(\sE_C)=H^0(\sK_C)$. Now, the key point is the following: $\sE_C\cong\sO_{\pp^1}(-s)\oplus\sO_{\pp^1}(s)$ if and only if there exist at most $s$ linearly independent global sections that vanish at a given point $p\in\pp^1$.

So let $C=\langle p,q\rangle$ be such that $\dim(U^{\circ}_{p}\cap U_{q})=s$. Therefore $U^{\circ}_{p}\cap U_{q}=\langle\lambda_1,\dots,\lambda_s\rangle\subset W\cong\mathbb{C}^{\oplus 4k+2}$. Since $J(U_{q},\lambda_i)=0$ by construction and  $J(U_{p},\lambda_i)=0$ by hypothesis, we have $J(U_{y},\lambda_i)=0$ for any $y\in C$ by linearity. In other words,
$A^t_CJ\lambda_i=0\in\mathbb{C}^{\oplus 2k}$ so $\lambda_i\in H^0(\sK_C)$ for all $i$. Moreover, $\lambda_i(q)=0\in U^{\circ}_{q}/U_{q}$ so we conclude by the previous remark.
\end{proof}

As a direct consequence of Theorem \ref{splitting-type-from-monad} we are able to obtain the following result.
\begin{corollary}\label{cBoundJump}
A smooth conic $C=\langle p,q \rangle$ induces a splitting of type $\sO_{\pp^1}(-s)\oplus\sO_{\pp^1}(s)$ if and only if the rank of the $2k\times 2k$ matrix $A^t(p)JA(q)$ is $s$. In particular, the splitting type of a conic is bounded by $2k$.
\end{corollary}

In what follows we describe the behaviour of 't Hooft instanton bundles when restricted to a conic $C$.

\begin{proposition}\label{pJumping}
Let $\sE$ be a special, charge $k$ instanton bundle. Let us consider the sections $s_i\in H^0(\sE(h_i))$. Let $Y^i=(s_i)_0$ be their respective vanishing locus and let $S_{(1,1)}$ be their dependence sextic del Pezzo surface. Then one of the following holds
\begin{itemize}
    \item [i)] If $C$ intersects $Y^i$ in $r>0$ points, then $\sE_C \cong \sO_{\pp^1}(1-r) \oplus \sO_{\pp^1}(r-1).$
    \item [ii)] If $C$ is an irreducible component of $Y^i$ then $\sE_C \cong \sO_{\pp^1} \oplus \sO_{\pp^1}$.
    \item [iii)] If $C$ is the support of a multiple conic, of multiplicity $\alpha\ge 2$, of an irreducible component of $Y^i$, then $\sE_C \cong \sO_{\pp^1}(-1) \oplus \sO_{\pp^1}(1)$.
    \item [iv)] If $C$ does not intersect $Y^1 \cup Y^2$, then $\sE_C \cong \sO_{\pp^1}(-l) \oplus \sO_{\pp^1}(l)$ with $l\in\{0,1\}$ and for the generic conic $C$ it  holds that $\sE_C \cong \sO_{\pp^1} \oplus \sO_{\pp^1}$.
    \end{itemize}
Moreover, $i)$, $ii)$ and $iii)$ hold for any $h_i$-'t Hooft instanton.
\end{proposition}
\begin{proof}
Let us consider the short exact sequence
\begin{equation}\label{sSerre2}
0 \to \sO_F \to \sE(h_i) \to \sI_{Y^i\mid F}(2h_i) \to 0.
\end{equation}
If $C$ is a conic that meets $Y^i$ in a finite number of points (counted with multiplicities) $p_1,\dots,p_r$, then tensoring sequence \eqref{sSerre2} by $\sO_C$ returns
\[
0 \to \sO_C \to \sE_{C}(h_i)\to \sO_{C}(2q -\sum_{i=1}^{r} p_i)\oplus \bigoplus_{i=1}^{r} \sO_{p_i} \to 0.
\]
In this case $\sE_{C}(h_i)$ splits as the direct sum $\sO_{\pp^1}(2-r) \oplus \sO_{\pp^1}(r)$, thus $$\sE_C \cong \sO_{\pp^1}(1-r) \oplus \sO_{\pp^1}(r-1).$$
Suppose now that $C$ is the reduced structure of a connected component $\tilde{C}$ of $Y^i$ and consider the curve $\Delta := Y^i\setminus \tilde{C}$. Thus we have the following short exact sequence
$$
0 \to \sI_{Y^i} \to \sI_{\tilde{C}} \to \sO_\Delta \to 0.
$$
Since $\Delta$ and $\tilde{C}$ are disjoint, restricting the above short exact sequence to $C$ yields
$$
\sI_{\tilde{C}}\otimes\sO_C \cong \sI_{Y^i} \otimes \sO_C.
$$
Thanks to Theorem \ref{theor-equiv}, $\tilde{C}$ is a primitive extension of $C$ and its ideal can be described as in equality \eqref{IdealMultConic}. In particular $\sI_{\tilde{C}}\otimes\sO_C$ is the relative conormal bundle of $\tilde{C}$, which is isomorphic to $\sO_{\pp^1} \oplus \sO_{\pp^1}(-2)$. Thus, restricting sequence \eqref{sSerre2} to $C$, we deduce $\sE_C \cong \sO_{\pp^1}(-1)\oplus \sO_{\pp^1}(1)$ if $C$ is the reduced structure of a multiple conic of $Y^i$. If $C$ is a simple conic of $Y^i$ then an analogous argument yields $\sE_C \cong \sO_{\pp^1}^{\oplus 2}$ since the conormal bundle of $C$ is $\sO_{\pp^1}(-1)^{\oplus 2}$.

It remains to consider the case when $Y^1 \cup Y^2$ and $C$ are disjoint. Firstly, using the same argument as in item $i)$, we see that the splitting type is bounded by one. In order to see that we have trivial splitting type for the generic conic $C$, let us consider a generic conic $C'\subset F$. Let $t_i\in H^0(\sO_F(h_i))$ be two global sections that define $C'$. $C'$ will intersect $S_{(1,1)}$ in a couple of points $\{p_1,p_2\}$. Now consider a general section $s=\alpha t_2 s_1 +\beta t_1 s_2$ as it was constructed in Proposition \ref{prop-geometriaE11}, $\Delta:=(s)_0$ and the curve $\Upsilon\subset S_{(1,1)}$ defined as the zero locus of $s_{\mid S_{(1,1)}}$. Let us observe that $(t_i)_0\cap Y^{i}$ is always contained in $\Upsilon$. Now if we consider two generic points $p,q\in\Upsilon\setminus ((t_1)_0\cap (t_2)_0)$, the unique conic $C$ passing through $p,q$ will satisfy $C\cap \Upsilon=\{p,q\}$ and $C\cap C'=\emptyset$. Therefore $\mathrm{length}(C\cap\Delta)=2$ and the restriction of the exact triple
$$
0 \to \sO_F \to \sE(h) \to \sI_{\Delta\mid F}(2h) \to 0
$$
to $C$ shows that $\sE_C\cong \sO_C\oplus\sO_C$.

\end{proof}

\begin{remark} 
The maximal splitting type of a special instanton bundle of charge $k$ on a conic is $2k-1$. Consider the special instanton constructed from a reduced curve $Y_{\bar{a}}$ with $\ell(\bar{a})=k$. Thanks to Theorem \ref{tClassificationSpecials} the dependence sextic is the union $S_{(1,0)} \cup S_{(0,1)}$ of two cubic surfaces. Consider the conic $C$ realized as the complete intersection of these cubics. Using Lemma \ref{lRestrictionPic} we see that the class of $C$ is $C_0+f$ inside both cubic surfaces. By computing the intersection number of $C$ and $Y_{\bar{a}}$ we obtain $2k$ points, thus the result follows from Proposition \ref{pJumping}.
\end{remark}

\bibliographystyle{amsplain} 
\bibliography{bibliographytHooft.bib}

\end{document}